\documentclass{amsart}

\usepackage[utf8]{inputenc}

\usepackage{setspace}
\usepackage{amsmath,amsfonts,amscd,amssymb,bm}
\usepackage{tikz-cd}
\usepackage{tkz-euclide}
\usepackage{mathtools}
\usetikzlibrary{positioning,arrows,calc}
\usepackage{comment}
\usepackage[pdfencoding=auto,colorlinks, linkcolor=red, anchorcolor=blue, citecolor=green]{hyperref}
\usepackage[all]{xy}
\usepackage{tikz}
\usepackage{enumitem}

\definecolor{myred}{RGB}{183,18,52}
\definecolor{myyellow}{RGB}{254,213,1}
\definecolor{myblue}{RGB}{0,80,198}
\definecolor{mygreen}{RGB}{0,155,72}

\let\emph\relax 
\DeclareTextFontCommand{\emph}{\bfseries\em}

\newcommand{\mA}{\mathcal{A}}

\newcommand{\mC}{\mathcal{C}}
\newcommand{\mD}{\mathcal{D}}
\newcommand{\mE}{\mathcal{E}}
\newcommand{\mF}{\mathcal{F}}

\newcommand{\mJ}{\mathcal{J}}

\newcommand{\mL}{\mathcal{L}}

\newcommand{\mV}{\mathcal{V}}
\newcommand{\mX}{\mathcal{X}}

\newcommand{\CC}{\mathbb{C}}

\newcommand{\NN}{\mathbb{N}}

\newcommand{\RR}{\mathbb{R}}

\newcommand{\ZZ}{\mathbb{Z}}

\newcommand{\into}{\hookrightarrow}

\newcommand{\mS}{\mathcal{S}}

\newcommand{\mT}{\Omega}
\newcommand{\CP}[1]{\mathbb{CP}^{#1}}

\newcommand{\w}{\omega}

\DeclareMathOperator{\id}{id}

\DeclareMathOperator{\UU}{U}
\DeclareMathOperator{\Sp}{Sp}

\DeclareMathOperator{\Emb}{Emb}

\DeclareMathOperator{\Conf}{Conf}
\DeclareMathOperator{\Diff}{Diff}
\DeclareMathOperator{\Symp}{Symp}

\DeclareMathOperator{\Aut}{Aut}

\DeclareMathOperator{\ev}{ev}

\DeclareMathOperator{\cod}{cod}
\DeclareMathOperator{\PD}{PD}
\DeclareMathOperator{\diag}{diag}

\newtheorem{thm}{Theorem}[section]
\newtheorem{dfn}[thm]{Definition}
\newtheorem{cor}[thm]{Corollary}
\newtheorem{lma}[thm]{Lemma}
\newtheorem*{lma*}{Lemma}
\newtheorem{prp}[thm]{Proposition}

\newtheorem{rmk}[thm]{Remark}
\newtheorem{conj}[thm]{Conjecture}

\begin{document}

\title[Stability of symplectomorphism groups]{Stability of the symplectomorphism groups of rational surfaces}

\author{ S{\'{\i}}lvia Anjos}
\email{sanjos@math.tecnico.ulisboa.pt}
\address{Center for Mathematical Analysis, Geometry and Dynamical Systems\\ Department of Mathematics\\  Instituto Superior T\'ecnico\\  Av. Rovisco Pais, 1049-001 Lisboa\\ Portugal}
\thanks{The first author is partially supported by FCT/Portugal through projects UID/MAT/04459/2020 and PTDC/MAT-PUR/29447/2017.}

\author{Jun Li}
\email{jli004@udayton.edu}
\address{Department of Mathematics\\  University of Dayton\\ Dayton, OH 45469}
\thanks{The second author is supported by AMS-Simons Travel Grant.}

\author{Tian-Jun Li}
\email{lixxx248@umn.edu}
\address{School of Mathematics\\  University of Minnesota\\ Minneapolis, MN 55455}
\thanks{The third author is supported by NSF Grant.}

\author{Martin Pinsonnault}
\email{mpinson@uwo.ca}
\address{Department of Mathematics\\  University of Western Ontario\\Canada}
\thanks{The fourth author is partially supported by NSERC Discovery Grant RGPIN-2020-06428.}



\begin{abstract}
We apply Zhang's almost K\"ahler Nakai-Moishezon theorem and Li-Zhang's comparison of $J$-symplectic cones to establish a stability result for the symplectomorphism group of a rational $4$-manifold $M$ with Euler number up to $12$. As a corollary, we also derive a stability result for the space of embedded symplectic balls in $M$. A noteworthy feature of our approach is that we systematically explore various spaces and groups associated to a symplectic cohomology class $u$ rather than with a single symplectic form $\w$. To this end, we prove a weaker version of the tamed $J$-inflation procedures of D. McDuff and O. Buse that fixes a gap in their original formulations.   
\end{abstract}

\maketitle

\setcounter{tocdepth}{2}

\tableofcontents

\section{Introduction}\label{Intro}

Let $M$ be a closed, oriented, smooth  4-manifold and let $\w$ be a symplectic form on $M$. When endowed with the standard $C^{\infty}$-topology, both the group $\Diff^+(M)$ of orientation-preserving diffeomorphisms and its closed subgroup $\Symp(M,\omega)$ of symplectomorphisms are infinite-dimensional Fr\'echet Lie groups. The group $\Diff^+(M)$ acts naturally on the space $\mT_M$ of symplectic forms on $M$ that are compatible with the given orientation, while the connected component of the identity $\Diff_0(M)$ acts on the subspace $\mT_{[\w]}$ of symplectic forms that are cohomologous to $\w$. Moser's lemma implies that the later action is transitive on each connected component of $\mT_{[\w]}$. In particular, if $ \mT_{\w}$ denotes the component containing $\w$, the evaluation map defines the so-called Moser fibration
\begin{equation}\label{MoserFibration}
\Symp(M,\w)\cap  \Diff_0(M)\to \Diff_0(M) \to \mT_{\w}
\end{equation}
This fibration was investigated by P. Kronheimer in~\cite{Kro99} using Seiberg-Witten theory to show the existence of non-trivial elements in $\pi_*\mT_{\w}$ and $\pi_*\Symp(M,\w)$ for some algebraic surfaces of general type.
 
As was first observed by D. McDuff in~\cite{McD01}, in order to get insights on the homotopy type of $\Symp(M,\w)$, it is often more convenient to replace the space $\Omega_\w$ of symplectic forms isotopic to $\w$ by the homotopy equivalent space $\mA_{\w}$ of all almost complex structures that are compatible with some symplectic form in $\mT_{\w}$, and to study the homotopy long exact sequence associated with the maps
\begin{equation}\label{McDuffFibration}
\Symp(M,\w)\cap  \Diff_0(M)\to \Diff_0(M) \to \mA_{\w}
\end{equation}
Combined with the symplectic inflation technique, this approach is especially effective in the case $M$ is a rational surface due to the abundance of $J$-holomorphic curves. For instance, it was shown in~\cite{AM99, LP04, ALP, AP13, AE17, LL16, LLW16} that, for certain rational surfaces, the symplectic cone
\begin{multline*}
\mC_M = \{u\in H^2(M;\RR)~|~u \text{~is represented by a symplectic form}\\ \text{compatible with the orientation}\}
\end{multline*}
admits a decomposition into polygonal regions that determine entirely the homotopy type of $\Symp(M, \w)$. This phenomenon is known as \emph{stability} of symplectomorphism groups, and the corresponding regions are called \emph{stability chambers}. 

In this note, we establish a finer version of stability for all rational surfaces with Euler number $\chi\leq 12$, that is, for all symplectic $4$-manifolds diffeomorphic to $S^2\times S^2$ or to a $k$-fold blow-ups of $\CP{2}$ with $0\leq k\leq 9$.

\begin{thm}\label{FinerStabilityThm}
Let $M$ be a rational surface with $\chi\leq 12$. For each integer $n\geq 1$, or for $n=\infty$, the symplectic cone $\mC_M$ admits a partition into convex regions, called {level $n$ stability chambers}, such that
\begin{enumerate}[label=(\roman*)]
\item each level $n$ chamber is a convex polyhedron defined by finitely many inequalities;
\item given two symplectic forms $\w$ and $\w'$ whose cohomology classes belong to the same level $n$ chamber, $n\geq 2$, we have $\pi_i\Symp(M,\w)=\pi_i\Symp(M,\w')$ for all $0\leq i\leq  2n-3$. In particular,  $\Symp(M,\w)$ is weakly homotopy equivalent to $\Symp(M,\w')$ whenever $[\w]$ and $[\w']$ belong to the same level $\infty$ chamber.
\end{enumerate}
\end{thm}

This stability theorem encapsulates all the previous stability results for symplectomorphism groups of rational surfaces into a single coherent statement. For instance, the level $1$ chambers determine the anticanonical class $K_\w$ associated to a symplectic form~\cite{LL01}, and correspond to the $K$-symplectic cone of~\cite{LL01}. Similarly, the level 2 stability chambers can be characterized in terms of reduced classes and, as explained in~\cite{LL16, LLW16}, determine both $\pi_0 \Symp(M,\w)$ and $\pi_1 \Symp(M,\w)$. Finally, all the stability regions described in~\cite{AM99, LP04, P08i, AP13, AE17} correspond to level $\infty$ stability chambers. Examples of stability chambers for manifolds of small Euler numbers are further discussed in section~\ref{section:ChambersSmallEulerNumbers}. We note that although the stability theorem alone does not compute $\pi_i(\Symp(M,\w))$, it does offer significant insights to understanding the homotopy type of $\Symp(M,\w)$, which is already rather complicated when $\chi\geq 4$ (see~\cite{LP04, AP13, AE17} for $\chi=5,6,7$ respectively). 

To prove the stability theorem~\ref{FinerStabilityThm}, we collect together old and new results that yield a characterization of stability chambers in terms of embedded symplectic spheres in specific homology classes of negative self-intersections (see Definition \ref{def:level n chamber}). One new feature is that we systematically explore various spaces and groups associated with a symplectic cohomology class $u$ rather than with a single symplectic form $\w$. This works best for $4$-manifolds for which cohomologous symplectic forms are diffeomorphic\footnote{To our knowledge, the only known examples of cohomologous symplectic forms that are not diffeomorphic occur in dimensions $\geq 6$. It may be possible that no such examples exist in dimension~$4$.}, in particular for rational or ruled surfaces (cf.~\cite{LM96, LL01}, and the surveys~\cite{Ltjforms, Salamonsurvey}). In this situation, given a cohomology class $u\in\mC_M$, the group $\Diff_u(M)$ of diffeomorphisms fixing $u$ acts transitively on the space $\mT_u$ of symplectic forms cohomologous to $u$. {The isotropy group of the action  $\Diff_u(M) \curvearrowright \mT_u$  at $\w\in \mT_u$ is exactly $\Symp(M, \w)$, because the isotropy preserves $\w$ and $\Symp(M, \w)$ is a subgroup of $\Diff_u(M)$. Hence this action yields a fibration}
\begin{equation}\label{UFIB}
\Symp(M, \w)\to \Diff_u(M) \to \mT_{u}.
\end{equation}
As we will see, the above fibration~\eqref{UFIB} provides a more natural approach to Theorem~\ref{FinerStabilityThm} than the fibrations~\eqref{MoserFibration}  or~\eqref{McDuffFibration}, and allows us to streamline many of the core arguments.

Another important ingredient for Theorem~\ref{FinerStabilityThm} is the tamed $J$-inflation procedure of D. McDuff and O. Buse. However, as was recently noted by P. Chakravarthy, the proofs given in~\cite{McD01, Buse11} make the unwarranted assumption that for every $\w$-tame $J$ and every $J$-holomorphic curve $Z$, one can find a family of normal planes that is both $J$-invariant and $\w$-orthogonal to $TZ$. This is true only if $\w$ is compatible with $J$ at every point of $Z$. We prove here a weaker version of \cite[Lemma 3.1]{McD01}  and \cite[Theorem 1.1]{Buse11}  whose proof relies on Li-Zhang's comparison of $J$-symplectic cones~\cite{LZ09}.
\begin{lma}[Weak $b^+=1$ $J-$compatible inflation]\label{Weak Inflation b1+}
Let $M$ be a symplectic $4$-manifold with $b^+=1$. Given a compatible pair $(J,\w)$ and a $J$-holomorphic embedded curve $Z$, there exists a symplectic form $\w'$ compatible with $J$ such that $[\w']= [\w]+ t PD(Z)$, $t\in [0,\lambda)$ where $\lambda= \infty$ if $Z\cdot Z\ge0$ and $\lambda= \frac{\w(Z)}{(-Z\cdot Z)}$ if $Z\cdot Z<0$.
\end{lma}
This weaker $J$-inflation lemma is sufficient for the proof of Theorem~\ref{FinerStabilityThm}.  We point out that it is also sufficient for all known results on symplectomorphism groups of $4$-manifolds with $b^+(M)=1$ relying on $J$-inflation, filling possible gaps in previous papers\footnote{Recently, P. Chakravarthy and  M. Pinsonnault were able to restore the tamed version when $Z\cdot Z\leq 0$ (\cite{CP2019}).}.

In a slightly different direction, we observe that the stability theorem for symplectomorphism groups has immediate implications for the structure of other spaces of interest in symplectic topology. For example, in dimension $4$, recall that there is a close relation between i) the space $\Emb_{\w}(B^4(c), M)$ of symplectic embeddings of a standard ball $B^4(c)$ of capacity $c$ into $(M,\w)$, ii) the symplectomorphism group $\Symp(M,\w)$, and iii) the symplectomorphism group $\Symp(\widetilde{M}, \widetilde{\w}_c)$ of the symplectic blow-up of capacity $c$. If the space $\Emb_{\w}(B^4(c), M)$ is connected, and if the capacity $c$ is not too large, it was shown in~\cite{LP04} that there is a homotopy fibration
\[\Symp(\widetilde{M}, \widetilde{\w}_c) \to \Symp(M,\w) \to \Emb_{\w}(B^4(c), M)/\Symp(B^4(c))\]
In order to extend this result, the fibration~\eqref{UFIB} is particularly convenient as it provides a simple way to combine the techniques of~\cite{LP04, P08i} with the stability theorem~\ref{FinerStabilityThm}. For instance, it yields the following stability result for the embedding spaces $\Emb_{\w}(B^4(c), M)$.
\begin{thm}\label{StabilitySingleBall}
Let $(M,\w)$ denote a rational surface with $\chi(M)\leq 11$ and of capacity $c_M$. Then, given two capacities $0<c<c'<c_M$, the restriction map $r_{c',c}:\Emb_{\w}(B^4(c'), M)\to \Emb_{\w}(B^4(c), M)$ is at least $(2n-3)$-connected as long as the cohomology classes of the symplectic blow-ups $[\widetilde{\w}_{c'}]$ and $[\widetilde{\w}_c]$ belong to the same level $n\geq2$ stability chamber of $\widetilde{M}=M \# \overline{ \CC P^2}$. In particular, the weak homotopy type of $\Emb_{\w}(B^4(c), M)$ is stable under deformations of $c$ as long as the cohomology class $[\widetilde{\w}_c]$ varies within a level $\infty$ stability chamber of the blowup $\widetilde{M}$.
\end{thm}
By adapting some arguments of~\cite{LP04, P08i}, the homotopy type of the embedding space $\Emb_{\w}(B^4(c), M)$ can be explicitly determined in some special cases. In particular, when the capacity $c$ is small enough, we obtain the following description of $\Emb_{\w}(B^4(c), M)$.
\begin{cor}\label{cor:embeddings and stabilizers}
Let $(M,\w)$ denote a rational surface with $\chi(M)\leq 11$. The space of unparametrized balls $\Emb_{\w}(B^4(c), M)/\Symp(B^4(c))$ is weakly homotopy equivalent to $M$ whenever the capacity $c$ is smaller than the symplectic area of any embedded symplectic sphere of negative self-intersection in the blow-up $(\widetilde{M_c},\widetilde{\w}_c)$.
\end{cor}

We note that  similar results hold for symplectic embeddings of $k$ disjoint standard balls of capacities $c_1,\ldots, c_k$ as long as $\chi(M)\leq 12-k$.

We end this introduction with some comments on the condition $\chi(M)\leq 12$ or equivalently,  $M=\CC P^2\# k\overline{\CC P^2}$ with $k\leq 9$, that appears in the statements of Theorem ~\ref{FinerStabilityThm} and Theorem~\ref{main-u version}. On one hand, the proofs of these theorems use a variation of the inflation strategy in~\cite{McD01}, given in Lemma~\ref{Weak Inflation b1+}, that is tailored to the rational or ruled surfaces. This inflation procedure relies on the almost K\"ahler cone Theorem in~\cite{Zha17} which, at the moment of writing, is only known to hold when $\chi(M)\leq 12$.  On the other hand, another important ingredient for the proofs of Theorem ~\ref{FinerStabilityThm} and Theorem~\ref{main-u version} is that the set 
\begin{multline*}
\mS_{\w}^{\leq -3}=\{ A \in H_2(M,\ZZ) ~|~ A \cdot A \leq -3 \text{~and~}A\text{~is represented}\\ \text{by an embedded symplectic sphere}\}
\end{multline*} 
is finite when $\chi(M)\leq12$ (see Lemma \ref{finite}). However, if $\chi(M) > 12$ then the set $\mS_{\w}^{\leq -3}$ can be infinite, see Remark \ref{extra} for more details. It is worth pointing out that these facts are related to the bounded negativity conjecture and to the Nagata conjecture, see the introduction of~\cite{DLW18} and~\cite{Zha17} for more complete discussions.\\

\subsection*{Organization of the paper} Theorem~\ref{FinerStabilityThm} is proven by assembling a number of statements that are discussed in different sections of the paper. We therefore guide the reader by briefly outlining the main steps and pointing to the corresponding sections where they are treated in detail.

The cell structure of symplectic cones is given under Definition~\ref{def:level n chamber}. The characterization of level 1 chambers appears in Proposition~\ref{-1same}, while the description and the special role of level 2 chambers is explained in Proposition~\ref{-2same} and Proposition~\ref{level2}. The higher level chambers are then described in Section~\ref{section:Level>3Chambers}, culminating with Corollary~\ref{asr}. Together, these results prove the first part of Theorem~\ref{FinerStabilityThm}.

The first step in proving the second part of Theorem~\ref{FinerStabilityThm} is to describe a partition of spaces of compatible almost complex structures analogous to the partition of the symplectic cones. This is shown in Section~\ref{section:PartitionAlmostComplexStructures}. Then, the second part of Theorem~\ref{FinerStabilityThm} is reformulated as Theorem~\ref{main-u version} with the associated proof in Section~\ref{section:ProofMainTheoremSecondVersion}.

\subsection*{Convention.} Throughout the paper, $M$ is a closed, oriented, smooth 4-manifold and  $\omega$ is an orientation-compatible symplectic form. 
We will often identify an integral degree 2 homology class with an integral degree 2 cohomology class via Poincar\'e duality, and we use the dot product to denote various pairings.

\subsection*{Acknowledgements}
We are grateful to  Olguta Buse, Richard Hind, Weiwei Wu, and Weiyi Zhang for their helpful conversations. We thank Dusa McDuff for her careful reading of the first draft of this manuscript,  very constructive comments, and for pointing out an oversight in section~\ref{stra}.  We also thank the anonymous referees for their careful reading and many constructive comments.

\section{Symplectic spheres and cell decomposition of symplectic cones}\label{section:Level1and2Chambers}

\subsection{Symplectic embedded spheres}
Let $M$ be a closed, oriented, smooth $4-$manifold and $\Omega_M$ the space of orientation-compatible symplectic forms. The symplectic cone $\mC_M$ $\subset H^2(M;\RR)$ is the set of classes of orientation-compatible symplectic forms. Clearly, it is contained in the  positive cone
\[\mathcal P_M=\{e\in H^2(M;\RR)~|~ e\cdot e>0\}.\]

Let $\mS_{\omega}$ denote the set of homology classes of embedded $\omega$-symplectic spheres and $K_{\w}$ the symplectic canonical class. For any $A\in \mS_{\omega}$, by the adjunction formula,
\begin{equation}\label{AF}
K_{\omega}\cdot A=-A\cdot A -2.
\end{equation}
We introduce the following subsets of $S_\w$ as in~\cite{LL16}: let
\[\mS_{\omega}^{\geq n},  \quad \mS_{\omega}^{>n}, \quad   \mS_{\omega}^{n}, \quad  \mS_{\omega}^{\leq n},\quad  \mS_{\omega}^{< n}\]
be the subsets of $\w-$symplectic spherical classes of self-intersection $\geq n$, $>n$, $=n$, $\leq n$, $<n$ respectively. The set $S_{\w}^{-2}$ turns out to be very useful in the study of $\pi_0\Symp(M,\w)$ and $\pi_1\Symp(M,\w)$ where $M$ is a rational 4-manifold with $\chi(M)\leq 8$, (\cite{LL16}~\cite{LLW16}). In this paper we continue to investigate how the sets of symplectic spherical classes $S_{\w}^{\leq -n}$, $n> 2$, are related to higher homotopy groups of $\Symp(M,\w)$.\\

\begin{dfn}\label{sympu}
For $u\in \mC_M$, let  $\mT_{u}$ denote  the space of symplectic forms in the class $u$. 
\begin{itemize}
\item  Let $\mS_{u}=\cup_{\w\in \mT_{u}} \mS_{\w}$, and define the subsets $\mS_{u}^{\geq n}$, $\mS_{u}^{>n}$,  $\mS_{u}^{n}$, $\mS_{u}^{\leq n}$, and $\mS_{u}^{< n}$ accordingly. 
\item  Similarly, define $\mS_{M}$ and its subsets $\mS_{M}^{\geq n}$, $\mS_{M}^{>n}$,  $\mS_{M}^{n}$, $\mS_{M}^{\leq n}$, and $\mS_{M}^{< n}$ by taking unions over all symplectic classes $u\in\mC_M$. 
\item {For a rational or ruled surface $M$, for convenience let $\Symp(M, u)$ denote $\Symp(M,\w)$ for an arbitrary $\w\in \mT_u$.} 
	As remarked in the introduction above \eqref{UFIB}, the isomorphism type of $\Symp(M, u)$ as a topological group is well defined.
\end{itemize}
\end{dfn}
 
For a rational or ruled surface, we will show in Proposition~\ref{u=w} that  $\mS_u=\mS_{\w}$ for any $\w\in \mT_u$.

\subsection{The canonical class \texorpdfstring{$K_u$}{Ku} and the group \texorpdfstring{$\Diff_u(M)$}{Diffu(M)}}
Let $K\in H^2(M;\mathbb Z)$ and define the $K$-symplectic cone as
\begin{equation}\label{def:K-symplectic cone}
\mC_{M,K}= \{e\in \mC_M~|~e=[\w] \hbox{ with }  K_{\w}=K\}.
\end{equation}

\begin{lma}\label{Ku}
Let $M$ be a closed, smooth 4-manifold and let $u$ be a symplectic class in $\mathcal C_M$. Then all symplectic forms in $\mT_u$ have the same symplectic canonical class, which we denote by $K_u$. Moreover, the group $\Diff_u(M)$ preserves the canonical class $K_u$.
\end{lma}
\begin{proof}
The first statement follows directly from~\cite{Taubesmore95} when $b^+(M)>1$. When $b^+(M)=1$, it follows from Proposition 4.1 in~\cite{LL01} which says that $\mC_{M, K}\cap \mC_{M, K'} =\emptyset$ if $K\ne K'$. The second statement follows from the first since  $K_{u}=K_{\phi^*\w}=\phi^*K_{\w}=\phi^*K_{u}$ for any $\w\in \mT_u$ and any $\phi\in \Diff_u(M)$.
\end{proof}
 
Let $\Aut(H^2(M;\RR))$ be the group of automorphisms of $H^2(M;\RR)$ preserving the intersection form. For any pair of classes $(a, b)$ in $H^2(M;\RR)$, let $D_{(a, b)}\subset \Aut (H^2(M;\RR))$ be the subgroup of automorphisms that are induced by diffeomorphisms and that preserve the classes $a$ and $b$.

Let $\Diff_h(M)$ be the subgroup of diffeomorphisms acting trivially on homology. We have the inclusions $\Diff_0(M)\subset \Diff_h(M)\subset \Diff_u(M)$. Note that $\Diff_0(M)$ is a normal subgroup of $\Diff_h(M)$ since it is the identity component of a Lie group, and $\Diff_h(M)$ is a normal subgroup of $\Diff_u(M)$ since it is  the kernel of the action homomorphism $\Diff_u(M)\to \Aut (H^2(M;\RR))$. Let $\Symp_0(M, \w)$ be the identity component of $\Symp(M, \w)$  and let $\Symp_h(M, \w) = \Symp(M, \w)\cap\Diff_h(M)$. Similarly, we have inclusions $\Symp_0(M, \w)\subset \Symp_h(M, \w)$ and $\Symp_h(M, \w)\subset \Symp(M, \w)$ as normal subgroups. Clearly, for all $\w\in \mT_{u}$, $\Symp(M,\w) \subset \Diff_u(M)$ and the homological action of $\Diff_u(M)$ factors through  $\Diff_u(M)/\Diff_h(M)$. Likewise, the action of $\Symp(M, \w)$ on $\Aut(H^2(M;\RR))$ factors through $\Symp(M, \w)/\Symp_h(M, \w)$.

\begin{lma}\label{homological action} Let $M$ be a rational or ruled surface.
Then, for $u\in \mC_M$ and  any $\w \in \mT_{u}$, the image subgroups of the homomorphisms $\Diff_u(M)\to \Aut (H^2(M;\RR))$ and $\Symp(M,\w)\to \Aut (H^2(M;\RR))$   are both equal to $D_{(K_{u}, u)}$. In other words, 
\[\Diff_u(M)/\Diff_h(M)=D_{(K_u, u)}=\Symp(M, \w)/\Symp_h(M, \w).\]
\end{lma}

\begin{proof}
By Theorem 1.4 and  Proposition 4.14 in~\cite{LW12}, the homological action of $\Symp(M, \w)$ is the group $D_{(K_{\w}, [\w])}$.
Therefore, the homological action of $\Diff_u(M)$ contains  $D_{(K_{\w}, [\w])}$. On the other hand, the homological action of $\Diff_u(M)$ is contained in  $D_{(K_{\w}, [\w])}$ since $\Diff_u(M)$ preserves $u=[\w]$ by definition and preserves $K_u=K_{\w}$ by Lemma~\ref{Ku}.
\end{proof}

\subsection{Properties of the set \texorpdfstring{$\mS_{u}$}{Su} }
Recall that $\mS_{u}=\cup_{\w\in \mT_{u}}  \mS_{\w}$. Here are two observations that are fundamental in what follows. 

\begin{lma}\label{relative cone} Let $M$ be a symplectic $4$-manifold. 
\begin{enumerate}[label=(\roman*)]
\item For any symplectic class $u\in\mathcal C_M$ and any $\phi\in \Diff^+(M)$, we have $\mS_{\phi^*u}=\phi_*\mS_u$. 
\item Moreover, if $b^+(M)=1$, then for any pair of symplectic classes $u,u'\in\mathcal C_M$,
\[\mS_{u}\cap \mS_{u'} = \{S\in \mS_u~|~u'[S]>0\} = \{S\in \mS_{u'}~|~u[S]>0\}.\] 
In other words, $S\in\mS_u$ is also in $\mS_{u'}$ whenever $u'$ is pairing positively with~$S$.
\end{enumerate}
\end{lma}
\begin{proof} The first claim follows from $\mS_{\phi^*\w}=\phi_*\mS_{\w}$ and $\Omega_{\phi^*u}=\phi^*\Omega_u$. The second claim is a direct consequence of Theorem 2.7 in~\cite{DoL10}.
\end{proof}

It is clear that  $ \mS_{\w} \subset \mS_{u}$ for any $M$ and $\w\in \mT_{u}$. We will show that the reverse inclusion also holds for rational or ruled surfaces.

\begin{lma}\label{u=w}
Let $M$ be a rational or ruled surface. Then $\mS_{u}=\mS_{\w}$ for all $\w\in \mT_{u}$.
\end{lma}
\begin{proof}
We just need to  show $\mS_u\subset \mS_\w.$  Let $S\in \mS_u$. Then there is a form $\w'\in \mT_u$ and a $\w'$-symplectic sphere $C$ with $S=[C]$. Because the action of $\Diff_u(M)$ on $\mT_u$ is transitive, see~\cite{LM96, LL01}, there is $\phi\in \Diff_u(M)$ so that $\phi^*\w=\w'$. Then the image $\phi(C)$ is a $\w$-symplectic sphere. By Lemma~\ref{homological action}, there is a $\psi \in \Symp(M, \w)$ acting as the inverse of $\phi$ in homology. Clearly, $\psi(\phi(C))$ is an $\w-$symplectic sphere in the  class $[C]$.
\end{proof}

Recall that $\mS_M=\cup_{u\in\mC_M} \mS_u$. Consider its subsets $\mS^{\geq n}_M,\mS^{>n}_M, \mS^{ n}_M, \mS^{ < n}_M, \mS^{ \leq n}_M$ as in Definition \ref{sympu}.  Each class $S\in\mS_M$ defines a hyperplane in $\mC_M$ that cuts the symplectic cone into three regions according to whether $S$ evaluates positively, negatively, or vanishes on symplectic classes. The second statement of Lemma~\ref{relative cone} motivates the following definition of chambers in $\mC_M$. 

\begin{dfn}\label{def:level n chamber}
Let $M$ be a rational or ruled surface. Fix a subset $\mD \subset \mS_{M}^{\geq -n}$. The \emph{level $n$ chamber $\Sigma_{\mD,n}$} of the symplectic cone $\mC_M$ is defined by the system of linear inequalities 
\[
\Sigma_{\mD,n} =\{ u\in \mC_M \,| \, u\cdot A > 0 \ \forall A\in \mD \   \text{and}\  u\cdot B \le 0  \ \forall B \in \mS_{M}^{\geq -n}\setminus \mD  \}
\] 
provided this set is non-empty.
Similarly, the \emph{level $\infty$ chamber $\Sigma_{\mD,\infty}$} is defined~as 
\[\Sigma_{\mD,\infty}=\{ u\in \mC_M\,| \,u\cdot A > 0 \ \forall A\in \mD  \   \text{and}\  u\cdot B \leq 0 \  \forall B\in\mS_{M}\setminus \mD  \}\]
provided this set is non-empty.
\end{dfn}

Next we show that the chambers can be seen as maximal subsets of $\mC_M$ in which the symplectic spherical classes of negative self-intersection do not change. 

\begin{lma}
	Given $n\geq 1$ or $n=\infty$, a level $n$ chamber  can be equivalently defined as a  maximal subset made of classes $u$ in the symplectic cone $\mC_M$ for which the sets of symplectic embedded spheres $\mS_{u}^{\geq -n}$(when $n=\infty$ this becomes $\mS_u$) coincide.
	This is to say, $u$ and $u'$ have the same symplectic embedded spheres classes $\mS_u^{\geq -n}$   and $\mS_{u'}^{\geq -n}$ if and only if the signs of $u(S)$ and $u'(S)$ are the same for all $S\in\mS_M^{\geq -n}$ ($\mS_M$ when $n=\infty$).
\end{lma}
\begin{proof}
	The only if direction is immediate from the coincidence of symplectic embedded spheres $\mS_{u}^{\geq -n}$ in the maximal subset definition. 

	The if direction follows from   \cite[Theorem 2.7]{DoL10}. To give more details, \cite[Theorem 2.7]{DoL10} states that the relative symplectic cone where a class $A\in H_2(M, \ZZ)$ has an embedded representative is given by the linear inequality  $\{ u\in \mC_M| u\cdot A > 0 \},$ for a 4-manifold $M$ with $b^+=1$. In other words, $S$ is $u$-symplectic if and only if $u(S)$ is positive.  Hence if the signs of $u(S)$ and $u'(S)$ are the same for all $S\in\mS_M^{\geq -n}$ then  $\mS_{u'}^{\geq -n}$ coincides with $\mS_{u}^{\geq -n}$.
  
	Here is a direct translation of the two definitions:  we can take $\mD = \mS_{u}^{\geq -n}\subset \mS^{\geq -n}_M$ and consider the set    
 \[\{ u'\in \mC_M| \, u'\cdot A > 0 \ \text{for all} \ A\in \mS_{u}^{\geq -n}\ \text{ and}\  u'\cdot B \leq 0 \  \text{for all} \ B\in \mS_{M}^{\geq -n}\setminus \mS_{u}^{\geq -n}\}.\] 
 This is the maximal subset of $\mC_M$ consisting of classes $u'$ such that the  symplectic embedded sphere classes $\mS_{u'}^{\geq -n}$ coincide with $\mS_{u}^{\geq -n}$.

\end{proof}

Observe that,  the intersection of a chamber with any convex subset of $\mC_M$ is itself convex, because it is defined using linear inequalities. The goal of the next few sections is to characterize level $n$ chambers of rational or ruled surfaces. In particular, we will show that they correspond precisely to the stability chambers mentioned in Theorem~\ref{FinerStabilityThm}.

\begin{prp}\label{-1same} 
Let $M$ be a rational or ruled surface. 
\begin{enumerate}[label=(\roman*)]
    \item If $\mS^{-1}_u=\mS^{-1}_{u'}$ then $K_u=K_{u'}$. 
    \item The level 1 chambers coincide with the $K$-symplectic cones.
\end{enumerate}  
\end{prp}

\begin{proof} We start by proving the first statement using results from~\cite{LL01}. Introduce the set of exceptional classes
\[\mE =\{E\in H_2(M,\ZZ)~|~E\cdot E=-1\text{~and~} E \text{~is represented by a smooth sphere} \},\]
the subset of $K$-exceptional spherical classes
\[\mE_{K}= \{E\in \mE~|~E\cdot K=-1\},\]
and the subset of $\w$-exceptional spherical classes 
\[\mE_{\w}= \{E\in \mE~|~E\text{~is represented by an~}\w\text{-symplectic sphere}\}.\]
By the adjunction formula, $\mE_{\w}\subset \mE_{K_\w}$. It follows from Lemma 3.5 in~\cite{LL01} that  \begin{equation} \label{Esets}
\mE_{\w}=\mE_{K_\w},
\end{equation}
and  Theorem 4 of the same paper yields a characterization of the $K$-symplectic cone as
\begin{equation} \label{Kcone}
\mC_{M,K}= \left\{e\in \mathcal P_M~|~ e \cdot E > 0  \text{~for all~} E\in \mE_K\right\}.
\end{equation}
Notice that  $\mE_{\w}=\mS_{\w}^{-1}$ by definition. So we have $\mE_{K_u}=S^{-1}_u$ by Lemma~\ref{u=w} and equality~\eqref{Esets}. It follows from the characterization~\eqref{Kcone} that $\mC_{M,K_u}$ is determined by $S^{-1}_u$, namely, that $\mC_{M,K_u}=\mC_{M,K_{u'}}$ whenever $\mS^{-1}_u=\mS^{-1}_{u'}$. Therefore, by Proposition 4.1 in~\cite{LL01}, we conclude that $K_u=K_{u'}$.
  
To prove the second statement, consider two classes $u$ and $u'$ belonging to the same level 1 chamber $\mV^1_u\subset\mC_M$. In particular, $\mS^{-1}_u = \mS_{u'}^{-1}$ which, from the previous discussion, implies that $K_u=K_{u'}=:K$. Consequently, $u$ and $u'$ belongs to the same $K$-symplectic cone, showing that $\mV^1_u\subset \mC_{M,K}$. Conversely, if $u$ and $u'$ are both in $\mC_{M,K}$, then $\mS^{-1}_{u}=\mS^{-1}_{u'}=\mE_K$. To show that $\mS^{\geq-1}_{u}=\mS^{\geq-1}_{u'}$, pick any class $S\in\mS_u$ with $S\cdot S\geq 0$. Then, as the three classes $S$, $\PD(u)$, and $\PD(u')$ belong to the same component of the positive cone, the light cone lemma implies that $S\cdot \PD(u)>0$. By Lemma~\ref{relative cone}, it follows that $S\in \mS_{u'}$. Consequently, $\mS^{\geq-1}_{u}=\mS^{\geq-1}_{u'}$. 
We conclude that all classes in $\mC_{M,K}$ belong to the same level 1 chamber, that is, $\mC_{M,K}\subset\mV^1_u$.
\end{proof}

\begin{cor}\label{cor:convexity of chambers}
For $n\geq 1$ or $n=\infty$, the level $n$ chambers are convex subsets of~$\mC_M$.
\end{cor}
\begin{proof}
By definition, each level $n$ chamber belongs to a level $1$ chamber which, by Proposition~\ref{-1same} (ii), is a $K$-symplectic cone $\mC_{M,K}$. { The characterization~\eqref{Kcone} shows that $\mC_{M,K}$ is itself convex. Note that each level $n$ chamber is given by linear inequalities.  Linear regions of a convex set are also convex.  So any level $n$ chamber in $\mC_{M, K}$ is also convex.  }
\end{proof}

Now that we have characterized $1$ chambers, we would like to describe $2$ chambers. The following proposition is the first step in this direction. 

\begin{prp}\label{-2same} 
Let $M$ be a rational or ruled $4$-manifold. If $\mS^{\geq-2}_u=\mS^{\geq-2}_{u'}$ then  $\Diff_u(M)= \Diff_{u'}(M)$. 
\end{prp}
\begin{proof}
By Lemma~\ref{homological action}, we have
\[\Diff_u(M)/\Diff_h(M)=D_{(K_u, u)} \quad \hbox{and} \quad \Diff_{u'}(M)/\Diff_h(M)=D_{(K_{u'}, u')}.\]
So we just need to identify $D_{(K_u, u)}$ and $D_{(K_{u'}, u')}$. For this purpose,  we recall some notions and facts from~\cite{LW12}. Define the set of spherical homology classes
\[\mL=\{L\in H_2(M;\ZZ)~|~  L\cdot L=-2 \hbox{ and $L$ is represented by a smooth sphere}\},\] 
the subset of $K$-null spherical classes
\[\mL_{K}=\{L\in \mL~|~ L\cdot K=0\},\]
and the subset  of $(K, \alpha)-$null spherical classes 
\[\mL_{K, \alpha}= \{L\in \mL_{K}~|~ \alpha\cdot L=0\}\] 
for a class $\alpha\in \mathcal C_{M, K}$.
By Theorem 4.14 in~\cite{LW12} we know that  $D_{(K,\alpha)}$ is generated by the reflections along elements in $\mL_{K, \alpha} $. So it suffices to show that  $\mL_{K_u, u} = \mL_{K_{u'}, u'}$.

By Proposition 5.16 in~\cite{DLW18}, for a symplectic form $\w$ and   $A\in \mL_{K_{\w}}$,  $A\in  \mS^{-2}_\w $ if and only if  $[\w]$ pairs positively with $A$. Since $\mS_{\w}=\mS_{[\w]}$ by Lemma~\ref{u=w}, we have the disjoint union decompositions
\begin{equation}\label{Ldecomposition}
\mL_{K_u}=  \mS^{-2}_u \bigsqcup  \left(-\mS^{-2}_u\right) \bigsqcup \mL_{K_u, u}, \quad  \quad 
\mL_{K_{u'}}=  \mS^{-2}_{u'} \bigsqcup  \left(-\mS^{-2}_{u'} \right) \bigsqcup \mL_{K_{u'}, u'}.
\end{equation}
Our assumptions are  $\mS^{-1}_u=\mS^{-1}_{u'}$ and $\mS^{-2}_u=\mS^{-2}_{u'}$. 
So the first statement in this proposition implies that  $K_{u}=K_{u'}$. Hence  $\mL_{K_u} =\mL_{K_u'}$. Therefore, if in addition $\mS^{-2}_u=\mS^{-2}_{u'}$ then the two decompositions~\eqref{Ldecomposition} imply 
$\mL_{K_u, u} =\mL_{K_u', u'}$. 
 This concludes the proof of Lemma~\ref{-2same}. 
\end{proof}

 \subsection{Level 2 chambers and the normalized reduced symplectic slice}

For manifolds with Euler number $\chi\leq 12$, the level $2$ chambers are best described in terms of reduced symplectic classes and the normalized reduced symplectic slice. As these two notions will also be useful for the proof of Theorem~\ref{FinerStabilityThm}, we briefly recall their definitions here. The interested reader is referred to Section~2.2 in~\cite{LL16} for a more detailed exposition.

Let $M$ be a rational surface not homeomorphic to $S^2\times S^2$. A basis $\{A_1,\ldots,A_{k+1}\}$ of $H_2(M;\ZZ)$ is said to be standard if i) each class $A_i$ is represented by a smoothly embedded sphere, and ii) the intersection pairing is represented, in this basis, by the diagonal matrix $\diag(1,-1,\ldots,-1)$. Note that any identification $M\simeq M_k:= {\CC P^2}\# k\overline{\CC P^2}$ determines a standard homology basis $\{H, E_1, E_2, \cdots, E_{k}\}$. 

Fix a standard basis. A class $\nu H-\sum_{i=1}^k c_i E_i\in H^2(M;\RR)$ is called \emph{reduced with respect to this basis} if
\begin{align*}
\nu&>0,\text{~for~}k=0;\\
\nu&>c_1>0,\text{~for~}k=1;\\
\nu&>c_1+c_2\text{~and~}c_1\geq c_2>0, \text{~for~}k=2;\\
\nu&\geq c_1+c_2+c_3 \text{~and~} c_1\geq c_2 \geq \cdots \geq c_k>0, \text{~for~}k\geq3.
\end{align*}
A reduced symplectic class is a reduced class in $\mC_M$. 
The normalized reduced symplectic slice $P_k=P(M_k)\subset \mC_{M_k}$ is the subspace of reduced symplectic classes with $\nu=1$. We represent such a class by the vector $(1\,|\,c_1, \cdots, c_k)$ or simply by $(c_1, \cdots, c_k)$.

For the product $M_{S^2}=S^2\times S^2$ with the standard basis $\{B=[S^2\times pt], F=[pt \times S^2]\}$, a  class $bB+fF\in \mC_{M_{S^2}}$ is called reduced if $b\geq f\geq 0$. The normalized reduced symplectic slice $P_{S^2}=P(M_{S^2})\subset \mC_{M_{S^2}}$ is the subspace of reduced symplectic classes with $f=1$.

A rational $4$-manifold $M$ together with a choice of a standard basis of $H_2(M,\ZZ)$ is called a \emph{framed surface}. From now on, in order to simplify the exposition, we will often implicitly assume that such a framing is chosen.

We now summarize the properties of the normalized reduced symplectic slice $P(M)$ in the next two propositions. 
\begin{prp} \label{redtran}
Let $M$ be a framed rational surface.  Then
\begin{enumerate}[label=(\roman*)]
\item  A convex combination of (normalized) reduced classes is (normalized) reduced. 
\item A  reduced class is symplectic if and only if it has a positive square.
\item  For a reduced symplectic class $u \in P(M)$, its canonical class   $ K_{u}$ is 
\[K_0:=-3H +\sum^{k}_{i=1} E_i\] 
if  $M=M_k$, and it is $K_0:=-2B-2F $ if $M=M_{S^2}$.
\item  Every class in $\mC_M$ is equivalent to a unique normalized reduced symplectic class under the action of { $\Diff^+(M)\times \RR^+$}, where $\RR^+$ acts by rescaling. In other words, the normalized reduced symplectic slice $P(M)$ is a fundamental domain of $\mC_M$ under the action of { $\Diff^+(M)\times \RR^+$}. Moreover, this action preserves the chamber structure of $\mC_M$.
\end{enumerate}
\end{prp}

\begin{proof}
Part (i) follows directly from the definition, part (ii)  is proved in ~\cite{LiLi02}, part (iii) is proved in~\cite{LL01}, and part (iv)  follows from ~\cite{LL01} and~\cite{GaoHZ} for rational classes, and from~\cite{KK17} for real classes (see also  the Math Review of~\cite{KK17}).
\end{proof}

\begin{dfn}\label{sethomclass} Let $M$ be a framed rational surface and let $K_0$ be the canonical class defined as in Proposition~\ref{redtran}~(iii). We set 
\[\mS_{(K_0)}=\bigcup_{u\in P(M)} \mS_u.\]
\end{dfn}

\begin{prp}\label{nrsc} 
Let $M$ be a framed rational surface with $\chi\leq 12$. 
\begin{enumerate}[label=(\roman*)]
\item The normalized reduced slice $P(M)$ is a convex region in $\RR^{\chi-3}$. Moreover, for $M_k$ with $0\leq k\leq9$, the class $-\frac{1}{3}K_0$ is in the closure of $P(M)$. The same holds for the class $-\frac{1}{2}K_0$ on $S^2\times S^2$.

\item For $M=M_k$, $3\leq k\leq8,$ the set $P(M_k)$ is naturally identified with the polyhedron in $\RR^k$ whose top vertex is $T_k=( \frac{1}{3}, \cdots, \frac{1}{3})$ and whose base is the convex hull of the following $k$ points in the hyperplane $\{c_k=0\}$:
\[G_1=(0,\cdots,0), \ G_2=(1, 0,\cdots, 0), \  G_3=\left(\frac{1}{2}, \frac{1}{2}, 0, \cdots,0\right),\]
\[G_4=\left( \frac{1}{3}, \frac{1}{3}, \frac{1}{3}, 0,\cdots,0\right), \ \cdots ,\ G_k=\left(\frac{1}{3},\cdots, \frac{1}{3}, 0\right).\]
When $k=9$, $P(M_9)$ has a similar description except that the top vertex $T_k=-\frac{1}{3}K_0$ is not in  $P(M_9)$. 
  
\item   When $3\leq k\leq 9$, consider the $K_0$-null spherical classes
\begin{equation}
  l_1= H-E_1-E_2-E_3,  \quad l_2 = E_1-E_2,\quad  \cdots, \quad l_k = E_{k-1}-E_k.  \label{simroot}
\end{equation}
Then the symplectic classes on the edge $T_kG_i$ are characterized by the property of pairing positively with $l_i$ and trivially with $l_j$ for all $j\ne i$. Consequently, the reduced symplectic classes are characterized as the symplectic classes which are positive on each $E_i$ and non-negative on each $l_i$.
 \end{enumerate}
\end{prp}
    
\begin{proof}
 All the statements can essentially be found in  Proposition 2.21 and Section 2.2.5 in~\cite{LL16}. 
The only case not explicitly covered there is $M=M_9$. In this case, the only difference comes from the following observation: for $M=M_9$, a reduced class still has non-negative square, and a normalized reduced class has square 0 if and only if it is $T_9=-\frac{1}{3}K_0$. Hence by parts (i) and (ii) of Proposition~\ref{redtran}, the normalized reduced symplectic slice $P(M_9)$ is convex and contains all the normalized reduced classes except $T_k=-\frac{1}{3}K_0$.
\end{proof}

Finally, we can characterize level 2 chambers in terms of normalized reduced symplectic slices.

\begin{prp}\label{level2}
Let $M$ be a framed rational surface.
\begin{enumerate}[label=(\roman*)]
\item Under the action of { $\Diff^+(M)\times \RR^+$}, every level 2 chamber in $\mC_M$ is equivalent to a unique level 2 chamber inside $P(M)$.

\item If  $\chi(M) \leq 12$, then $P(M)$ is a polyhedron whose open faces   are precisely the level $2$ chambers in $P(M)$. In particular, any level $2$ chamber in $P(M)$ and $\mC_M$ is defined by finitely many linear inequalities. 
\end{enumerate}
\end{prp}
\begin{proof}  {The first statement follows from part (iv) of Proposition~\ref{redtran}.  More precisely, the action of $\Diff^+(M)$ amounts to the Cremona algorithm, which sends any symplectic cohomology class to a unique reduced one; and the $\RR^+$ action is the normalization process.  Notice that both actions preserve the positivity of symplectic area on any given curve class, that is, $u\cdot S= \phi^*u\cdot\phi_*S$. In particular, this holds for the classes of symplectic (-2) spheres. Hence any level $2$ chamber in $\mC_M$ is sent to a unique level 2 chamber inside $P(M)$.}

{The second statement follows from part (iii) of Proposition~\ref{nrsc}, which states that the level 2 chambers and the hyperplane defined by  $\w(E_i)=0$ are the boundaries of $P(M)$. To see this in more detail, we just point out that by definition,  the hyperplane defined by $\w(E_i)=0$ does not belong to $P(M),$ and  $\w(l_i)=0$ does belong to $P(M).$   Hence the open faces of $P(M)$ are exactly level 2 chambers.   We further remark that higher-level chambers may cut a level 2 chamber. But this does not affect the statement here, and it just means that the open faces also have chamber structures.}
\end{proof}

\section{The set \texorpdfstring{$\mS_{(K_0)}^{\leq -3}$}{Sk>=-3} and the level \texorpdfstring{$n$}{n} chambers}\label{section:Level>3Chambers}

In this section we show that for rational surfaces with $\chi(M)\leq12$ the level $n$ chambers, $n\geq 3$ or $n=\infty$, are convex polyhedrons with finitely many facets. We start with a simple observation. 

\begin{lma} \label{linearcut}  Let $M$ be a framed rational surface. Let $n\geq 1$ or $n=\infty$.
\begin{enumerate}[label=(\roman*)]
\item Under the action of {$\Diff^+(M)\times \RR^+$}, every level $n$ chamber in $\mC_M$ is equivalent to a unique level $n$ chamber inside $P(M)$. 
\item If $\chi(M)\leq 12$, then for any  $S\in \mS_{(K_0)}^{\leq -3}$, the hyperplane in $ H^2(M;\RR)$ which vanishes on $S$ cuts the normalized reduced symplectic slice $P(M)$ into two non-empty regions.
 \end{enumerate}
\end{lma}

\begin{proof} The first statement again follows from Proposition~\ref{redtran}. 
Regarding the second statement, for a class $S\in \mS_{(K_0)}^{\leq -3}$, the adjunction formula implies that 
\[ S\cdot (-K_0)=2+S\cdot S \leq -1.\] 
Observe that by Proposition \ref {nrsc}, the class $T_k=-\frac{1}{3}K_0$ is in the closure of $P(M)$. On the other hand, by the definition of  $\mS_{(K_0)}$ (see Definition \ref{sethomclass}), $S\in \mS_u$ for some  $u\in P(M)$. For such a  class $u$,  we must have $u\cdot S>0$. Then the hyperplane in $ H^2(M;\RR)$ which vanishes on $S$ divides $P(M)$ into two non-empty regions: one containing $u$ and one containing points near to $T_k$.
\end{proof}
 
\subsection{Finiteness of the set \texorpdfstring{$\mS_{u}^{\leq -3}$}{Sk<=-3}}

The following finiteness result will be useful. 

\begin{lma}\label{finite}
Let $M$  be a rational surface with  $\chi(M)\leq 12$. Then $\mS_{u}^{\leq -3}$ is a finite set for all $u\in \mathcal C_M$.
\end{lma}

\begin{proof} 
By Proposition~\ref{redtran} and Lemma~\ref{relative cone} we can  assume that $u$ is reduced (and normalized). 

This is easy for $M=S^2\times S^2$.
So we assume that $M={ \CC P^2}\# k\overline{ \CC P^2},  k \leq 9$.
Let \[ S=aH +\sum b_i E_i \in \mS_{u}^{\leq -2}. \]
It follows from  the adjunction formula that 
\begin{equation}\label{adjunction} a^2-3a+2 =\sum b^2_i+\sum b_i.
\end{equation}

We first observe that for a fixed $a$, there are only finitely many choices of vectors $(b_1, ..., b_k)$ satisfying~\eqref{adjunction}. This follows from the fact that 
the non-negative quadratic function $f(x)=x(x+1)$ defined on the integers $\mathbb Z$ is proper, that is, the inverse image of any finite interval is a finite set. In particular, there are only finitely many classes $(a|b_1,\ldots,b_k)$ in $\mS_{u}^{\leq -3}$ with $a=0$. We next consider the two cases $a>0$ and $a<0$ separately.
\begin{itemize}
\item $a> 0$: By the observation above, it suffices to bound $a$ from above.  
Let $c^2=-S\cdot S \geq 3$. We then have
\begin{equation}\label{square}
c^2+ a^2= b^2_1+ \cdots b^2_k.
\end{equation}
We rewrite~\eqref{adjunction} as
\begin{equation}\label{adjunction'}
2 - c^2 -3a = b_1+\cdots +b_k.
\end{equation}
Applying  Cauchy-Schwartz  to~\eqref{adjunction'}, together with ~\eqref{square}, we have
\[(c^2+ 3a - 2)^2 \leq  k(b_1^2+\cdots +b_k^2)= k(c^2+ a^2) \leq  9(c^2+ a^2). \]
This can be written as
\begin{equation}\label{c3a4}
6a(c^2-2)\leq  -(c^4-13 c^2 +4)=-(c^2-\frac{13}{2})^2+ \frac{169}{4}-4\leq 39.
\end{equation}
Since $c^2\geq  3$, equation~~\eqref{c3a4} gives us the bounds $0<a \leq  \frac{39}{6} <7$.

\bigskip
\item $a<0$: In this case, it follows from Lemma 3.4 in~\cite{Chen20} that
\begin{equation}\label{negative H}
\{S \in\mS_{(K_0)}^{\leq -2} ~|~ a< 0\}=\{ -pH +(p+1)E_1+\sum_{j=2 }^kb_j E_{j}, \quad p > 0,\quad b_j=0 \hbox{ or } -1 \}.
\end{equation}

Write 
\[u=H-c_1E_1-\sum_{i=2}^k c_iE_i \]
with $0<c_i<c_1$. 
Then for $p>0, b_j=0$ or $-1$ for $j\geq 2$ and 
\[S=-pH +(p+1)E_1+\sum_{j=2}^k b_jE_{j}\in \mS_{u}^{\leq -2}\subset \mS_{(K_0)}^{\leq -2},\]
since $u\cdot E_j>0$, we have 
\begin{eqnarray} \label{list}  \nonumber u\cdot S &=&u\cdot (-pH+(p+1)E_1)+ u\cdot (\sum_{j=2}^k b_j E_{j})\\ \nonumber
&\leq& u\cdot (-p H+(p +1)E_1)\\ \nonumber
&=&-p+(p+1)c_1 \\ \nonumber 
&\leq&  0 \nonumber 
\end{eqnarray}
if $p$ is sufficiently large.
It follows that $\{S\in \mS_{(K_0)}^{\leq -2}~|~a<0\}$ is a finite set. 
\end{itemize}
\end{proof}

We call a point in $\mC_M$ a \emph{rational point} if every coordinate of the point is rational, otherwise we call it an \emph{irrational point}.

\begin{cor}\label{asr}
Let $M$ be a framed rational surface with $\chi(M) \leq 12$ and let $n\geq 1$ or $n=\infty$.
\begin{enumerate}[label=(\roman*)]
\item Each level $n$ chamber is a convex polyhedron defined by a finite set of linear inequalities. 
\item Except for the monotone chamber which consists of the single point $-\frac13K_0$, each level $n$ chamber has positive dimension. In particular, there are infinitely many (indeed dense) rational points and irrational points.
\end{enumerate}
\end{cor}
\begin{proof}
The first statement follows directly from Proposition~\ref{level2} and Lemma~\ref{finite}, which say, respectively, that each level 2 chamber is a convex region defined by finitely many linear inequalities and that there are only finitely many elements in $\mS_u^{\leq -3}$.  
  
The second statement is a corollary of the first one, since any level $\infty$ chamber is obtained by cutting the interior or a facet of the reduced slice by finitely many hyperplanes.
\end{proof}
 
\begin{rmk}\label{extra}
We mention a couple of facts that are not needed for the proof of Theorem~\ref{FinerStabilityThm}.
The set $\mS_{u}^{-2}$ is finite if $\chi(M)\leq 11$. For classes $(a\,|\,b_1,\ldots,b_k)$ with $a\leq0$ this follows from the proof of Lemma~\ref{finite}.
When $a>0$, we can argue as follows. If $c^2=2$ and  $M=M_k$, $k\leq 8$, then by adding $b_i=0$ for $k+1\leq i\leq 9$, we have
\[-3a=\sum_{i=1}^9 b_i, \quad a^2-\sum_{i=1}^9 b_i^2=-2.\]
Note that $b_9=0$, so we can apply Cauchy-Schwartz to the vectors $(b_1,\hdots,b_9)$ and $(1,\hdots,1,0)$ in order to obtain 
\[9a^2 =  \left(  \sum_{i=1}^9 b_i \right)^2 \leq 8 \sum_{i=1}^9 b_i^2 = 8(a^2+2), \]
which implies that  $a \leq 4$.
  
However, the set $\mS_{(K_0)}^{-2}$ is infinite when $\chi(M)\geq12$. For instance, for $\chi(M)=12$, there are 72 classes of type $\pm(E_i-E_j)$ and 168 classes of type $\pm(H-E_i-E_j-E_k)$, $i,j,k$ distinct. It is not hard to see that adding a multiple of $K_0=3H-E_1-\cdots-E_9$ to any of these classes defines another class of self-intersection $-2$ and pairing $0$ with K that is represented by an embedded sphere. Moreover, all such classes can be obtained this way. Furthermore, these classes generate an infinite reflection subgroup $\Aut(H_2,K_0)\subset\Aut(H_2(M_9,\ZZ))$ that acts transitively on the infinite sets $S^{-1}_{(K_0)}$ and $S^{-2}_{(K_0)}$. 
{
It is easy to obtain infinitely many  $(-3)$ spheres, i.e.  $\mS_{(K_0)}^{-3}$ by blowing up the infinitely many  $(-2)$ spheres.  It immediately follows that  $\mS_{(K_0)}^{-3}$  and  $\mS_{(K_0)}^{\le -3}$ are infinite sets when $\chi(M)>12$.  See \cite[proof of Proposition 1]{Dem} for a complete discussion.}

\end{rmk}

\subsection{The level \texorpdfstring{$\infty$}{infinity} chambers for small Euler numbers}\label{section:ChambersSmallEulerNumbers}
We explicitly describe the reduced slice and the level $\infty$ stability chambers for rational symplectic manifolds of Euler numbers $\chi \leq 5$.

\subsubsection{$\CP{2}$}
For the complex projective space $\CP{2}$, the symplectic cone is the ray $\lambda H$, $\lambda>0$, and the normalized reduced
symplectic slice consists of a single point $\{H\}$, so that the cell decomposition is trivial. 

\subsubsection{$S^2\times S^2$} 
For the trivial bundle $S^2\times S^2$, the symplectic cone is the positive quadrant $\mu B+ \nu F$, $\mu,\nu>0$, and the reduced
symplectic slice consists of all classes with $\mu\geq1$ and $\nu=1$. For $u=\mu B+F$, let's write $\mu=\ell+\lambda$ with $\ell$ a nonnegative integer (possibly zero), and $\lambda\in(0,1]$. The set $\mS_{u}^{\leq 1}$ consists of all classes $B-kF$ with $1\leq k\leq\ell$. Note that all negative spherical classes have even self-intersection. For any finite $2k\geq 2$, the level $2k$ chambers are the intervals $[1,k]$ and $(k,\infty)$. The level $\infty$ stability chambers are the half-open intervals $(n,n+1]$, $n\geq 1$, which coincides with the stability regions described in~\cite{AM99}.

\subsubsection{$M_1$} 
For the non-trivial bundle $\CC P^2\#\overline{\CC P^2}\to \CP{1}$, the symplectic cone is made of classes $\alpha H + \beta E$ with $\alpha>\beta>0$. Let $B=E$ and let $F=H-E$ be the class of a fiber. The reduced symplectic slice can be identified with the set of classes of the form $\mu B+F$ with $\mu>0$. As before, given a class $u=\mu B+F$, let's write $\mu=\ell+\lambda$ with $\ell$ a nonnegative integer (possibly zero), and $\lambda\in(0,1]$. The set $\mS_{u}^{\leq 1}$ consists of all classes $B-kF$ with $0\leq k\leq\ell$. This time, all negative spherical classes have odd self-intersection. For any finite $2k+1\geq 1$, the level $2k+1$ chambers are the intervals $(0,k]$ and $(k,\infty)$. The level $\infty$ stability chambers are the half-open intervals $(n,n+1]$, $n\geq 0$, which again coincides with the stability regions described in~\cite{AM99}.

\subsubsection{$M_2$}
In this case, the level $\infty$ chambers can be read off from  Theorems 1.1 and 1.6 in~\cite{LP04}, where the homotopy type of $\Symp(M_2)$ is determined. The reduced symplectic slice can be identified with the set of classes of the form $H-c_1E_1-c_2E_2$ with $0<c_2\leq c_1<c_1+c_2<1$. We illustrate it in Figure~\ref{fig:2bl}. Note that the interior walls in the picture, that is, the lines in the interior of the triangle, come in pairs as they are determined by existence of embedded symplectic spheres in classes
\[ kE_1-(k-1)H  \quad\text{and}\quad kE_1-(k-1)H-E_2 \]
with $k >1$, and which are examples of classes considered in Lemma~\ref{finite}. They divide the reduced slice into level $\infty$ chambers.

\begin{figure}[ht]
\begin{center}
\begin{tikzpicture}
[
scale=8,
axis/.style={->,thick},
vector/.style={-stealth,black,very thick},
]
\coordinate (O) at (0,0,0);(-0.1,0,0)node{$O$};
\draw[axis] (0,0,0) -- (1,0,0) node[anchor=north east]{$c_1$};
\draw[axis] (0,0,0) -- (0,.5,0) node[anchor=north west]{$c_2$};
\draw[vector]  (0.4,0.4,0)--(0,0,0) node[anchor=north ]{$c_1=c_2$};
\draw[dashed,black] (0.4,0.4,0)--(0.8,0,0);
\draw[dashed,white] (0,0,0)--(0.8,0,0);
\node[draw] at(0.8,-0.05,0){$A:(1,0)$};
\node[draw] at(0.4,0.45,0){$B:(\frac12,\frac12)$};
\node[circle, fill, inner sep=0.2pt] at(0.53,0.27,0){1};
\node at(0.71,0.13,0)  at(0.71,0.13,0){$\ddots$};
\node[circle, fill, inner sep=0.2pt] at(0.4,0,0){1};
\node[draw] at(0.4,-0.05,0) {$1/2$};
\node[circle, fill, inner sep=0.2pt] at(0.61,0.19,0){1};
\node[draw] at(0.53,-0.05,0) {$2/3$};
\draw[black] (0.61,0.19,0)--(0.53,0,0);
\node[circle, fill, inner sep=0.2pt] at(0.53,0,0){1};
\node[draw] at(0.53,-0.05,0) {$2/3$};
\node at(0.69,-0.03,0) {$\cdots$};
\draw[black] (0.4,0.4,0)--(0.4,0,0);
\draw[black] (0.4,0,0)--(0.53,0.27,0);
\draw[black] (0.53,0.27,0)--(0.53,0,0);
\end{tikzpicture}
\end{center}
\caption{Stable chambers of the 2-point blowup of $\CC P^2$}
\label{fig:2bl}
\end{figure}

\subsubsection{$M_3$}
In this case, the homotopy groups of $\Symp(M, \w)$ are calculated  in \cite[Proposition 3.3]{AP13}. The level $\infty$ chambers can be read off and agree with part (ii) of Theorem~\ref{FinerStabilityThm} in this case. 
 
The polytope $OAM_2 M_3$ depicted in Figure~\ref{fig:3bl} is the reduced symplectic slice of $\CC P^2\# 3 \overline{ \CC P^2}$. It consists in classes $H-c_1E_1-c_2E_2-c_3E_3$ such that $0<c_3\leq c_2\leq c_1<c_1+c_2+c_3\leq 1$. The hyperplanes in the interior of the cone are the walls defined by the curves in classes $S \in \mS_{(K_0)}^{\leq -3}$, that is, the hyperplane in $H^2(M;\RR)$ which vanishes on a specific $S$ corresponds to a wall in $P(M)$. 
\begin{figure}[ht]
  \centering
  \includegraphics[scale=0.8]{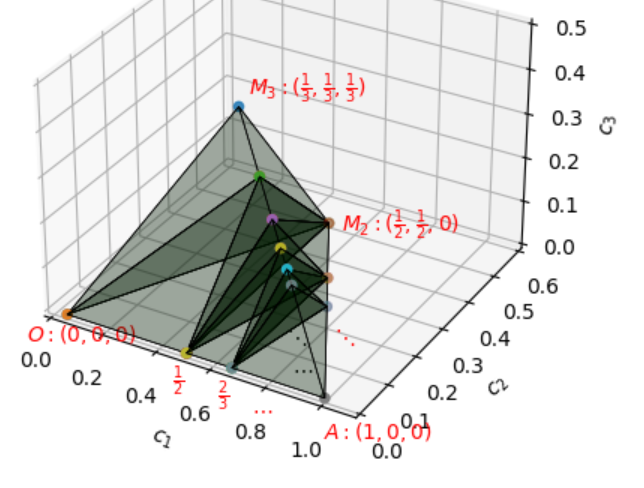}
  \caption{Stable chambers of the 3-point blowup of $\CC P^2$}
  \label{fig:3bl}
\end{figure}
Here we list the classes of the interior walls: the first wall from the left is defined by the sign of the class $E_1-E_2-E_3$; then as one moves right toward point $A$, for each integer $k>1$ there's a pattern of 4 walls determined by classes of the forms
\begin{align*}
kE_1-(k-1)H, &\quad  kE_1-(k-1)H-E_3, \\ kE_1-(k-1)H-E_2, &\quad kE_1-(k-1)H-E_2-E_3.
\end{align*}
They never intersect with each other except on the edges.

\section{Decomposition of \texorpdfstring{$\mA_u$}{Au}, Almost K\"ahler cone, and \texorpdfstring{$J$-inflation}{J-inflation}}\label{section:PartitionAlmostComplexStructures}

In order to relate the cell decomposition of the symplectic cone $\mC_M$ to the homotopy type of symplectomorphism groups, we  now turn our attention to spaces of almost complex structures. 

\subsection{The homotopy fibration \texorpdfstring{$\Symp(M, \w)\to \Diff_u(M) \to \mA_{u}$}{}}\label{stra}

It was first observed by D. McDuff in~\cite{McD01} that there is a homotopy equivalence between the space $\mT_{\w}$ of symplectic forms that are isotopic to a given symplectic form $\w$, and the space $\mathcal A_\omega$ of almost complex structures tamed by (or compatible with) at least one form in $\mT_{\w}$. This simple fact, together with the symplectic $J$-inflation procedure, underlie all known stability results on the homotopy type of symplectomorphism groups of $4$-manifolds. 

It turns out that for rational or ruled surfaces, the homotopy equivalence between $\mT_{\w}$ and $\mathcal A_\omega$ holds, more generally, for the spaces of forms and almost-complex structures associated to a symplectic class $u\in \mathcal C_M$. Let $\mT_{u}$ be the space of symplectic forms cohomologous to $u$, and let $\mA_{u}$ be the space of almost complex structures $J$ for which there exists some $\w\in \mT_{u}$ compatible with $J$, that is,
\[\mA_{u} = \{J~|~\exists~ \w\in\mT_{u},~J\text{~is compatible with~}\w\}\]
Both  $\mT_{u}$ and $\mA_{u}$ are infinite-dimensional Fr\'echet manifolds. 

\begin{lma}\label{convexfiber}
Let $M$ be a rational or ruled surface and let $u\in\mC_M$. The spaces $\mT_{u}$ and $\mA_{u}$ are homotopy equivalent. In particular, there is a canonical bijection between the sets of path-connected components of $\mT_{u}$ and $\mA_{u}$. Moreover, $\Diff_u(M)/\Diff_0(M)$ acts transitively on the sets of path connected components of $\mT_{u}$ and $\mA_{u}$.
\end{lma}
\begin{proof}
Consider the space $P_{u}$ of pairs
\[ P_{u} = \{(\w, J) \in \mT_{u} \times \mA_{u}| \, \w \ \text{is compatible with} \ J \}.\]
Since the projection  $\alpha_u:P_{u} \to  \mA_{u}$ is a fibration whose fiber at $J$ is the convex set of $J$-compatible symplectic forms, the projection is a homotopy equivalence. Likewise, the projection $\beta_u: P_{u} \to  \mT_{u}$ is also a homotopy equivalence since it is a fibration whose fiber at $\w$ is the contractible space of $\w$-compatible almost complex structures. Let $\gamma_u: \mT_{u}\to P_u$ be a homotopy inverse of $\beta_u$. Then $\alpha_u\circ \gamma_u$ is the desired homotopy equivalence between $\mT_{u}$ and $\mA_{u}$.

Of course, $\alpha_u\circ \gamma_u$ also induces a canonical bijection between the sets of connected components of $\mT_{u}$ and $\mA_{u}$. Recall that for rational or ruled surfaces, any two cohomologous symplectic forms are diffeomorphic. Hence, $\Diff_u(M)$ acts transitively on $\mT_{u}$, and there is a  transitive action of $\Diff_u(M)/\Diff_0(M)$ on the sets of path connected components of $\mT_{u}$ and $\mA_{u}$.
\end{proof}

The group $\Diff_u$ acts on the three spaces $P_u$, $\Omega_u$, and $\mA_u$. Clearly, the two projections $\alpha_u:P_u\to\mA_u$ and $\beta_u:P_u\to\Omega_u$ are equivariant. Although the evaluation map $\Diff_u(M)\to P_u$ is never a fibration (because compatible pairs $(\w,J)$ have local invariants such as the curvature of the associated metric and the rank of the Nijenhuis tensor $N(J)$), we can consider its homotopy fiber $\mF_{P_u}$ defined through the fibrant replacement\footnote{For a nice introduction to fibrant replacements and homotopy commuting diagrams, see~\cite{Sel97} Section 7.6.}
\[\mF_{P_u}\to\overline{\Diff_u(M)}\to P_u.\]
Composing this fibration with the two projections $\alpha_u,~\beta_u$, we obtain a commutative diagram of fibrations and homotopy equivalences
\[
\begin{tikzcd}[row sep=small]
\Symp(M,u)\arrow[r]\arrow[d] & \Diff_u(M) \arrow[d,hookrightarrow,"\simeq"]\arrow[r] & \Omega_u \\
\mF_{\Omega_u} \arrow[r]               & \overline{\Diff_u(M)} \arrow[r]& \Omega_u \arrow[u,equal]\\ 
\mF_{P_u} \arrow[r]\arrow[u] \arrow[d] & \overline{\Diff_u(M)} \arrow[d, equal]\arrow[u,equal]\arrow[r] & P_u \arrow[d,"~\alpha_u"]\arrow[u,"~\beta_u" right]\\ 
\mF_{\mA_u}\arrow[r] & \overline{\Diff_{u}(M)} \arrow[r]      & \mA_{u}\\
 & \Diff_u(M) \arrow[u,hookrightarrow,"\simeq"]\arrow[ru] & 
\end{tikzcd}
\]
Then, the 5-lemma applied to the induced ladder of long exact sequences implies that the homotopy fiber $\mF_{\mA_u}$ of the evaluation map $\Diff_u(M)\to\mA_u$ is weakly homotopy equivalent to $\Symp(M,u)$.  {Note that the compatibility of the action of $\Diff_u(M)$ on $\mT_u$ is proved at the end of the proof of Lemma \ref{convexfiber}. } The long exact homotopy sequence associated to the maps
\begin{equation} \label{homotopy fibration}
\Symp(M, u)\into \Diff_u(M) \to \mA_{u}.
\end{equation}
will be investigated through a decomposition of $\mA_{u}$ arising from negative spherical classes in $\mS_{u}^{\leq-2}$.

\subsection{A partition of \texorpdfstring{$\mA_{u}$}{Lg} along spherical classes \texorpdfstring{in $ \mS_{u}^{\leq-2}$}{}}
In~\cite{LL16}, J. Li and T. J. Li introduced a partition of $\mJ_{\omega}$ into submanifolds characterized by the existence of embedded $J$-holomorphic spheres of self-intersection at most $-2$. We introduce an analogous decomposition for $\mA_{u}$.

For each $A\in  \mS_{u}^{\leq -2}$ we associate { the positive} even integer
\begin{equation}\label{codimension}
{cod_A}=2(-A\cdot A-1).
\end{equation}

\begin{dfn}\label{def:admissible set}
Let $u$ be a symplectic class. Given a finite subset $\mD\subset \mS_{u}^{\leq -2}$,
\[\mD=\{A_1,\cdots ,A_n ~|~  A_i\cdot A_j \geq 0  \text{~if~}  i\neq j\},\]
define the codimension of the  set ${\mD}$ as $cod({\mD})= \sum_{A_i\in \mD } cod_{A_i}$. We call such a set $\mD$ an \emph{admissible subset of $S_{u}^{\leq -2}$}. 
{In the case ${\mD}=\emptyset$ we define $cod(\emptyset)=0$.  }

\end{dfn}

Notice that $cod({\mD})$ is a non-negative integer and $cod(\mD)\leq cod(\mD')$ whenever $\mD\subset \mD'$.

\begin{dfn}\label{def:prime subset}
Given an admissible set $\mD\subset\mS_{u}^{\leq -2}$ as above, we define the associated subset  $\mA_{u,\mD}\subset\mA_u$ as
\begin{multline}
\mA_{u, \mD}=\{ J\in \mA_{u}~|~\text{a class~} A\in \mS_u^{\leq -2}  \text{~has an embedded~} J\text{-holomorphic}\\ \text{representative \emph{if, and only if,~}} A\in \mD\}.
\end{multline}
whose  $cod (\mA_{u, \mD})$ is defined to be equal to $cod (\mD)$.
\end{dfn}

Note that $J\in \mA_{u, \mD}$ if, and only if, the admissible set $\mD$ is precisely the set of all spherical classes $A\in\mS_u^{\leq -2}$ that admit embedded $J$-holomorphic representatives. It follows that we have the disjoint union decomposition
\[\mA_{u} =\bigsqcup_{\text{admissible~}\mD} \mA_{u, \mD}\]
Note also that the prime subset $\mA_{u,\emptyset}$ associated to the empty set consists of compatible almost-complex structures for which there are no embedded $J$-holomorphic spheres $S$ of self-intersection $S\cdot S \leq -2$. It follows from Gromov-Witten theory that this set is always open.
\begin{rmk}
It may seem more natural to include exceptional classes $S\in\mS^{-1}_u$ in the definition of prime subsets. However, as these classes are generic in the sense of Gromov-Witten theory, the sets of almost-complex structures for which a collection $\mD\subset\mS^{-1}_u$ is represented by $J$-holomorphic embedded spheres is always open.
\end{rmk}
We introduce a filtration $\{ \mA_{u}^2\subset\cdots\subset\mA_{u}^{2n}\subset \mA_{u}^{2n+2}\subset\cdots \}$ of $\mA_u$,
where
\begin{equation}\label{Adec}
\mA_{u}^{2n}=\bigsqcup_{\substack{\text{admissible~}\mD\\cod(\mD)<2n}} \mA_{u, \mD}.
\end{equation}

{Note that $ \mA_{u}^2 $ is just  $ \mA_{u,\emptyset}$.}

Let $\mX_{u,2n}$ denote  the  complement of $\mA_{u}^{2n}$. Clearly, it can also be written as a \emph{disjoint} union
\begin{equation}
\mX_{u,2n}=\bigsqcup_{\substack{\text{admissible~}\mD\\ \cod(\mD)\geq 2n}} \mA_{u, \mD}.\end{equation}
It is actually useful to express $\mX_{u,2n}$ as the following  (not necessarily disjoint) union
\begin{equation}\label{Udec}
\mX_{u,2n}=\bigcup_{\substack{\text{admissible~}\mD\\ cod(\mD)\geq 2n}} U_{u, \mD},
\end{equation}
where
\[ U_{u, \mD}=\{ J\in \mA_{u}~|~  A\in \mS_u \text{~has an embedded~} J\text{-hol representative {\bf whenever}~}  A\in \mD\}.\]
Clearly, $\mA_{u, \mD} \subset U_{u, \mD} $ and
\[U_{u, \mD}= \bigsqcup_{\substack{\text{admissible~}\mD'\\ \mD\subset \mD'}}\mA_{u, \mD'}.\]

We have the following analogue of \cite[Proposition 7.1]{AP13} whose proof is the same mutatis mutandis, the only difference being that we consider the larger space $\mA_{u}$ instead of $\mJ_{\omega}$. 

\begin{prp}\label{p:stratum}
Each subset  $U_{u,\mD}$ is a co-oriented Fr\'echet submanifold of $\mA_{u}$ of (real) codimension $cod(\mD)$. 
It follows that $\mX_{u, 2n}$ is the union of submanifolds of codimension at least $2n$. \qed
\end{prp}

\begin{rmk}
The analogue of Proposition 2.14 in~\cite{LL16} is also valid for $\mA_{u, \mD}$, that is, if $\chi(M) \leq 12$ then $\mA_{u, \mD}$ is a  submanifold of $\mA_u$ with  codimension $cod({\mD})$. The proof is similar.
\end{rmk}

\subsection{The almost K\"ahler cone}
Recall the two notions of $J$-symplectic cones, the $J$-tame cone and the $J$-compatible cone:
\begin{align}\label{J symplectic cones}
\mathcal K_J^{t}&=\{[\omega]\in H^2(M;\RR)~|~\omega \text{~tames~} J\},\\
\mathcal K_J^{c}&=\{[\omega]\in H^2(M;\RR)~|~\omega \text{~is compatible with~} J\}.\notag
\end{align}

$\mathcal K_J^{c}$ is also called the almost K\"ahler cone.
Both $\mathcal K_J^c$ and $\mathcal K_J^t$ are convex cohomology cones contained in the positive cone $\mathcal P_M=\{e\in H^2(M;\RR)~|~e\cdot e >0\}.$ We say that an almost complex structure $J$ is almost K\"ahler if $\mathcal K_J^c \neq 0$. Clearly,  $\mathcal K_J^c\subset \mathcal K_J^t$.  And for an almost K\"ahler $J$ on a 4-manifold with $b^+=1$, they are equal.
\begin{thm}[Corollary 1.1 in \cite{LZ09}]\label{c=t}
Let $(M,J)$ be an almost complex 4-manifold.  If $b^+(M)=1$ and $\mathcal K_J^c\ne \emptyset$, then $\mathcal K_J^c=\mathcal K_J^t$.
\end{thm}

{Let $C_J(M)$ be the curve cone of $(M, J)$ where $J$ is  almost
K\"ahler, i.e. $\mathcal K_J^c\ne \emptyset$ as in Theorem \ref{c=t}:}
\[C_J (M) = \left\{\sum a_i[C_i]~|~a_i>0,~  C_i \text{~is an irreducible~}J\text{-holomorphic subvariety} \right\}.\]
Here an irreducible $J$-holomorphic subvariety is the image of a $J$-holomorphic map $\phi:\Sigma \to M$ from a complex connected curve $\Sigma$, where $\phi$ is an embedding off a finite set. Let $\overline{C}_J(M)$ be the closure of $C_J(M)$. 

Let $\overline{C}_J^{\vee, >0}(M)$ be the positive dual of $\overline{C}_J(M)$ under the homology-cohomology pairing. Clearly, $\mathcal K_J^t\subset \overline{C}_J^{\vee, >0}(M)$ since the integral of a $J$-tamed symplectic form over a $J$-holomorphic subvariety is positive. Motivated by the famous Nakai-Moishezon-Kleiman criterion in algebraic geometry which characterizes the ample cone in terms of the (closure of) curve cone for a projective $J$, and the recent K\"ahler version\footnote{Established by Buchdahl~\cite{Buch99} and Lamari~\cite{Lam99} in dimension 4, and by Demailly-Paun~\cite{DP04} in arbitrary dimension.} of the Nakai-Moishezon criterion (in dimension $4$), which characterizes the K\"ahler cone in terms of the curve cone for a K\"ahler $J$, one asks whether there is a tamed/almost K\"ahler version of the Nakai-Moishezon criterion.
Such a theorem was proven in~\cite{Zha17} for rational surfaces with $\chi(M)\leq 12$.

\begin{thm}[Theorem 1.6 in~\cite{Zha17}]\label{ccinf}
Let $M$ be a rational surface with $\chi(M) \leq 12$.
For an almost K\"ahler structure $J$ on $M$, the positive dual of the closure of the curve cone is  the almost K\"ahler  cone, i.e. 
\[\overline{C}_{J}^{\vee,>0}(M)=K_{J}^{c}(M).\]
\end{thm}

\subsection{A remark on applying \texorpdfstring{J}{J}-inflation to 4-manifolds with \texorpdfstring{$b^+=1$}{Lg}}\label{tcinf}
An important ingredient for Theorem~\ref{ccinf} is the tamed $J$-inflation by  McDuff~\cite{McD01} and Buse~\cite{Buse11}. Note that the proofs make the unwarranted assumption that for every $\w$-tame $J$ and every $J$-holomorphic curve $Z$ one can find a family of normal planes that is both $J$ invariant and $\w$-orthogonal to $TZ$. This is true only if $\w$ is compatible with $J$ at every point of $Z$. We state here a weaker version of \cite{McD01}~Lemma 3.1 and \cite{Buse11}~Theorem 1.1 that includes this extra hypothesis.

\begin{thm}\label{weaker version}
Given a compatible pair  $(J,\w)$ on $M^4$, one can inflate along a $J$-holomorphic curve $Z$, so that there exists a symplectic form $\w'$ taming  $J$ in the cohomology class
\[ [\w']= [\w]+ t PD(Z),~ t\in [0,\lambda)\] 
where $\lambda= \infty$ if $Z\cdot Z\ge0$ and $\lambda= \frac{\w(Z)}{(-Z\cdot Z)}$ if $Z\cdot Z<0$.
\end{thm}

Once we have this weaker version of inflation,   Theorem \ref{c=t} grants that there is a symplectic form $\w''$ in the same cohomology class of $[\w']$ that is compatible with $J$, and we immediately obtain Lemma~\ref{Weak Inflation b1+} that we reproduce here for convenience.

\begin{lma*}[Weak $b^+=1$ $J-$compatible inflation]
Let $M$ be a symplectic $4$-manifold with $b^+=1$. Given a compatible pair $(J,\w)$ and a $J$-holomorphic embedded curve $Z$, there exists a symplectic form $\w'$ compatible with $J$ such that $[\w']= [\w]+ t PD(Z)$, $t\in [0,\lambda)$ where $\lambda= \infty$ if $Z\cdot Z\ge0$ and $\lambda= \frac{\w(Z)}{(-Z\cdot Z)}$ if $Z\cdot Z<0$.
\end{lma*}

This $J$-compatible inflation for $b^+(M)=1$ is sufficient for the proof of Theorem~\ref{ccinf}. It is also sufficient for all  the known  stability  results of $\Symp(M, \w)$ when $b^+(M)=1$ (see Section~\ref{compwise}).
 
\begin{rmk}\label{rmk:tamed inflation}
Recently, Chakravarthy and  Pinsonnault (\cite{CP2019}) were able to restore the tamed version of $J$-inflation as originally stated in~\cite{McD01} and~\cite{Buse11} for curves of non-positive self-intersection. This is also sufficient for the proof of Theorem~\ref{ccinf}. At the time of writing, it is still unclear whether the tamed $J$-inflation process along positive curves can be fully restored.

For the readers' convenience, we give a sketch of the proof of Theorem~\ref{weaker version}. We first recall that the framework of the proof of  \cite{Buse11}~Theorem 1.1 (the proof of~\cite{McD01}~Lemma 3.1 has exactly the same structure):  first take the symplectic normal bundle of $Z$,  then at $p\in Z$ one can take the tangent spaces $V_1$ and the normal complement $V_2$ in $T(M^4)$. One can find a form $\w'$  in the cohomology class
$[\w']= [\w]+ t PD(Z),~ t\in [0,\lambda)$ by adding up $\w|_Z$ and the Thom form of the symplectic normal bundle.  The key is then to show that $\w'$ tames the given $J$.  In particular, to show that at $p$ on $V_1, V_2$, $J$ is  a block matrix  $J_p = \begin{pmatrix} 
A & B \\
C & D 
\end{pmatrix}$, which allows to estimate $\w(-, J_p(-))$   using the local form.  If one had $B=C=0$, then all the  estimates in (2.7)- (2.11) of \cite{Buse11} (or equivalent parts of \cite{McD01}) are all valid.  However, only assuming that $(J, \w)$ is a tame  pair (as in \cite{Buse11} and \cite{McD01})  cannot guarantee this. One needs the stronger assumption that  $(J, \w)$ is a compatible pair as in Theorem \ref{weaker version}.  Then the fact that  $(J,\w')$ is a tame pair follows from the estimates in (2.7)- (2.11) of \cite{Buse11}.
\end{rmk}

\section{Stability of \texorpdfstring{$\Symp(X,\w)$}{Lg}}
In this section, we prove the second part of Theorem~\ref{FinerStabilityThm}. We also outline an alternative approach for its proof that uses the Moser fibration. Finally, we relate the topology of $\Symp(M,\w) $ with the topology of the space of symplectically embedded balls.

Note that from the description of the chambers in terms of symplectic spherical classes of negative self-intersection, Part (ii) of Theorem~\ref{FinerStabilityThm}  can be restated as 
 
\begin{thm}\label{main-u version} 
Let $M$ be a rational surface with $\chi(M)\leq 12$ and  $u, u'\in \mathcal C_M$. If $\mS_{u}^{\geq -n}=\mS_{u'}^{\geq -n}$ for $n\geq 2$, then $\pi_i\Symp(M, u)=\pi_i\Symp(M, u')$ for $0\leq i\leq  2n-3$. In particular, if $\mS_{u}=\mS_{u'}$, then $\pi_i\Symp(M, u)=\pi_i\Symp(M, u')$ for all $i\geq 0$.
\end{thm}

We will prove this version of the theorem.

\subsection{Proof of Theorem~\ref{main-u version}}\label{section:ProofMainTheoremSecondVersion}
{
First note that for two symplectic forms $\w_1$ and $\w_2$ in the same cohomology class $u$ of a rational surface $M$, $\w_1$ and $\w_2$ are diffeomorphic by \cite{McD98}.  From Moser's argument we know that $Symp(M,\w_1)$ and $Symp(M,\w_2)$ are isomorphic as topological groups.    Hence we will use $Symp(M,u)$ as an abstract Lie group defined up to homeomorphism,  when the topological type of the group is concerned, as defined in the third bullet of Definition \ref{sympu}.}

To prove any stability result for symplectomorphism groups, we have to relate $\Symp(M,u)$ and $\Symp(M,u')$ for some $u,u'\in\mC_M$ in some natural way. For this purpose, we may consider the fibrations
\[
\begin{tikzcd}[row sep=small]
\Symp(M,u) \arrow[r] & \Diff_u \arrow[d, equal]\arrow[r] & \Omega_u \\ 
\Symp(M,u') \arrow[r] & \Diff_{u'} \arrow[r] & \Omega_{u'}
\end{tikzcd}
\]
However, there is no canonical way to define a map between $\Omega_u$ to $\Omega_{u'}$. On the other hand, there exist natural inclusions between various subspaces of $\mA_u$ and $\mA_{u'}$. Consequently, as was first observed by D. McDuff in~\cite{McD01}, by replacing the spaces of symplectic forms by the spaces of compatible almost-complex structures, one obtains comparison maps between homotopy groups of $\Symp(M,u)$ and $\Symp(M,u')$ using the long exact homotopy sequence associated to~\eqref{homotopy fibration}. 
\begin{lma}  \label{inclusion} Let $M$ be a rational surface with $\chi(M) \leq 12$.
If  $\mS_{u}\subset \mS_{u'}$ then  $\mA_{u}\subset \mA_{u'}$.
\end{lma}
\begin{proof}
Let $J\in \mA_{u}$, i.e.  $J$ is compatible with some $\w\in \mT_{u}$, and let $S_J$ be the set of classes of embedded $J$-holomorphic rational curves. Then clearly $S_J\subset \mS_{\w}=\mS_u$ for any $\w$ taming $J$. Since $\mS_{u}\subset\mS_{u'}$, it follows that $u'$ is positive on $S_J$. Then, by Theorem~\ref{ccinf},  we conclude that $u'$ is in the almost K\"ahler cone of $J$. In other words, $J\in \mA_{u'}$. This shows $\mA_{u}\subset \mA_{u'}$.
\end{proof}

\begin{lma} \label{level n}
Let $M$ be a rational surface with $\chi(M) \leq 12$. If $\mS_{u}^{\geq -n}=\mS_{u'}^{\geq -n}$,  with $n \geq 2$, then $\mA_{u}^{2n} =\ \mA_{u'}^{2n}$.
\end{lma}
\begin{proof}
We first observe that if $\mD\subset \mS^{\leq -2}_u$ with  $cod(\mD)<2n$, then $\mD\subset \mS_{u}^{\geq -n}$. Since $\mS_{u}^{\geq -n}=\mS_{u'}^{\geq -n}$, the  decompositions of $\mA_u^{2n}$ and $\mA_{u'}^{2n}$ in ~\eqref{Adec} are indexed by the same set of admissible subsets of  $S^{\geq -n}_u$, all with codimension less than $2n$. We will show that, for each such $\mD$, we have $\mA_{u, \mD}=\mA_{u', \mD}$.

We take any such admissible subset $\mD$. If $J\in \mA_{u, \mD}$, then $J$ is compatible with some $\w\in \mT_{u}$ and the {\bf only} $J$-holomorphic embedded spheres of self-intersections $\leq -2$ represent the homology classes in $\mD$. Since $u'$ is positive on all classes in $\mD$, it follows from Theorem~\ref{ccinf} that $u'$ is in the almost K\"ahler cone of $J$, i.e. $J$ is compatible with some symplectic form $\w'$ with class $u'$. In other words, $J\in \mA_{u', \mD}$.  This means that $\mA_{u, \mD}\subset \mA_{u', \mD}$. The same strategy applies to prove the converse.
\end{proof}

\begin{lma}\label{basic}
Let $M$ be a rational surface with $\chi(M) \leq 12$. If $\mA_{u}\subset \mA_{u'}$,  $\mA^{2n}_{u}\subset \mA^{2n}_{u'}$ and  $\mS_{u}^{\geq -n}=\mS_{u'}^{\geq -n}$, $n\geq 2$, then $\pi_i\Symp(M,u)=\pi_i\Symp(M, u')$ for $0\leq i\leq 2n-3.$
\end{lma}

\begin{proof}
Firstly, since  $\mS_{u}^{\geq -n}=\mS_{u'}^{\geq -n}$, $n\geq 2$ implies $\mS_{u}^{\geq -2}=\mS_{u'}^{\geq  -2}$, by Proposition~\ref{-2same}, $\Diff_u(M) = \Diff_{u'}(M)$.  We then consider the homotopy commuting diagram
\begin{equation}\label{hcomm}
\begin{tikzcd}[row sep=small]
\Symp(M,u) \arrow[r]\arrow[d] & \Diff_u \arrow[d, equal]\arrow[r] & \mA_u\arrow[d,hookrightarrow] \\ 
\Symp(M,u') \arrow[r] & \Diff_{u'} \arrow[r] & \mA_{u'}
\end{tikzcd}
\end{equation}
By Proposition~\ref{p:stratum}, the complement of the inclusion $\mA_{u}\subset \mA_{u'}$ is a union of submanifolds of codimension $\geq 2n$. Then, by standard transversality arguments, the inclusion induces an isomorphism  $\pi_i(\mA_{u})\to \pi_i(\mA_{u'})$  for $0\leq i\leq 2n-2$. Therefore, from the homotopy commuting diagram \eqref{hcomm} and the associated ladder of homotopy long exact sequences, we obtain isomorphisms $\pi_i \Symp(M,u)\to \pi_i\Symp(M,u')$ for $0\leq i\leq 2n-3$.
\end{proof}

We now have established the following weaker version of Theorem~\ref{main-u version}. 

\begin{prp}\label{2n-3}
Let $M$ be a rational surface where $\chi(M)\le 12$. If $\mS_{u}\subset \mS_{u'}$ and $\mS_{u}^{\geq -n}=\mS_{u'}^{\geq -n}$, for some $n\geq 2$, then $\pi_i\Symp(M, u) =\pi_i\Symp(M, u')$ for $0\leq i\leq  2n-3$. Consequently, if $\mS_{u}=\mS_{u'}$ then $\pi_i\Symp(M, u)= \pi_i\Symp(M, u')$ for all $i \in \NN$.
\end{prp}
\begin{proof}
It follows from Lemma~\ref{inclusion}, Lemma~\ref{level n} and Lemma~\ref{basic}.
\end{proof}

To prove Theorem~\ref{main-u version}, it still  remains to remove the condition $\mS_{u}\subset \mS_{u'}$. We will actually reduce the general case to the case $\mS_{u}\subset \mS_{u'}$.

\begin{proof} [Proof of Theorem~\ref{main-u version}]
By Proposition~\ref{redtran}, we can assume both $u$ and $u'$ belong to the same normalized reduced slice whose description is given in Proposition~\ref{nrsc}.

We first observe that there is nothing to prove in the case of the level $n$ monotone chamber since, as explained in Corollary~\ref{asr}, it only contains the point $-\frac{1}{3}K_0$.

Next, we observe that for symplectic classes $u$ and $u'$ in the same level $\infty$ chamber one has $\pi_i\Symp(M, u)=\pi_i\Symp(M, u')$ for all $i\geq 0.$  This follows immediately from Proposition~\ref{2n-3}  since, in the level $\infty$  chamber, we have $\mS_{u}=\mS_{u'}$ and hence the assumptions of the proposition are satisfied.
  
Now, assume that neither $u$ nor $u'$ is the monotone class, and that they belong to the same level $n$ chamber.  
In order to apply Proposition~\ref{2n-3}, we need to understand the symmetric difference between $\mS_u$ and $\mS_{u'}$, that is,
\[ \mS_u\,\triangle \,\mS_{u'}=\left(\mS_u\setminus \mS_{u'}\right)\cup \left(\mS_{u'}\setminus \mS_u\right).\]

Since $\mS_{u}^{\geq -n}=\mS_{u'}^{\geq -n}$, it follows that
\begin{equation}\label{-n-1}
\mS_{u} \triangle \mS_{u'}\subset  \mS_{u}^{\leq -(n+1)}\cup \mS_{u'}^{\leq -(n+1)}\subset \mS_{(K_0)}^{\leq -(n+1)},
\end{equation}
and, by Lemma~\ref{finite}, then $\mS_{u} \triangle \mS_{u'}$ is a finite set $\{S_1, \cdots, S_k\}$. {Note that for all  $i=1, \ldots, k,$  $u \cdot S_i$ and  $ u'\cdot S_i$ are both non-zero since $u$ and $u'$ are not on the wall defined by $S_i$, and they have different signs.  The reason is that by definition $S_i$ is symplectic (hence pairs positively) for one class (say $u$ without loss of generality), meanwhile not symplectic for the other ($u'$ in this case), and by \cite[Theorem 2.7]{DoL10} $u'$ pairs negatively with $S_i$.  }

By Corollary~\ref{cor:convexity of chambers}, the line segment $\overline{uu'}$ connecting $u$ and $u'$ lies in the level $n$ chamber of $u$ and $u'$.
We say the line segment $\overline{uu'}$ is \emph{generic}, if it is in general position with respect to the set of hyperplanes defined by classes in $\mS_{u} \triangle \mS_{u'}$. 
In particular, under this genericity condition on the segment $\overline{uu'}$, its endpoints $u$ and $u'$ belong to the interior of their respective level $\infty$ chamber and $\overline{uu'}$ intersects the hyperplanes defined by $\{S_1,\ldots,S_k\}$ at $k$ distinct points $p_i$. { Clearly, we can assume the points $p_i$ are ordered along line segment $\overline{uu'}$ from $u$ to $u'$.  Namely,  $u<p_1<\cdots<p_k<u'$.}

Assume for now that $\overline{uu'}$ is in generic position. 
The points $\{p_1,\ldots,p_k\}$ define $k+1$ intervals on $\overline{uu'}$. Pick  $u_0=u$ from the first interval, $u_{k}=u'$ from the last interval, and one point $u_i, 1\leq i \leq k-1, $ from the interior of each of the remaining  $k-1$ intervals. Then for each $0\leq i \leq k-1,$ we consider the pair of adjacent points $u_i$ and $u_{i+1}$. 
  
{The key observation is that $\mS_{u_i}$ and $\mS_{u_{i+1}}$ only differ by the class $S_{i+1} \in \mS_{u} \triangle \mS_{u'}\subset \mS_{(K_0)}^{\leq -(n+1)}$.} Consequently, 
we either have $\mS_{u_i}\subset \mS_{u_{i+1}}$ or $\mS_{u_{i+1}}\subset \mS_{u_{i}}.$ Without loss of generality, we assume that $\mS_{u_i}\subset \mS_{u_{i+1}}$. Then we have the inclusion $\mA_{u_i}\subset \mA_{u_{i+1}}$ by Lemma~\ref{inclusion}. { On the other hand it follows from formula \eqref{-n-1} that  $\mS_{u_i}^{\geq -n}= \mS_{u_{i+1}}^{\geq -n}$. Therefore  by Lemma~\ref{level n} we get  $\mA_{u_i}^{2n} =\ \mA_{u_{i+1}}^{2n}$.} Now all the assumptions of Proposition ~\ref{2n-3} are satisfied and we can now conclude  that $\pi_j(\Symp(M, \w_i))=\pi_j(\Symp(M, \w_{i+1}))$ for $0\leq j\leq 2n-3$, where $[\w_i]=u_i$ and $[\w_{i+1}]=u_{i+1}$. Since this holds for all pairs $(u_i,u_{i+1})$, this establishes Theorem~\ref{main-u version} in the generic case.

If the line segment $\overline{uu'}$ is not in generic position, then we can do a small perturbation of $u$ and $u'$ in the interior of their respective level $\infty$ chamber, making $\overline{uu'}$ generic without changing the homotopy groups of $\Symp(M,u)$ and $\Symp(M,u')$. To see this, first recall that Lemma~\ref{asr} states that any level $\infty$ chamber other than the monotone one contains dense sets of rational and irrational points. We then observe that the intersection points of the hyperplanes are all rational points, since they are solutions to linear equation systems with integer coefficients. Hence, we can perturb $u$ to a rational point and $u'$ to an irrational point within their own level $\infty$ chamber, to obtain a line segment $\overline{uu'}$ in generic position. 
\end{proof}

\begin{rmk}\label{12} We speculate that Theorem~\ref{main-u version} holds for arbitrary rational surfaces. 
However, there are difficulties to generalize the arguments to rational surfaces with $\chi\geq 13$. We mention a couple here.
\begin{enumerate}
    \item The almost K\"ahler Nakai-Moishezon criterion has not been established for these rational surfaces.
    \item Lemma~\ref{finite} is not valid for these rational surfaces and so the ``walls" are somewhere dense in the symplectic cone.
\end{enumerate}
\end{rmk}

\subsection{Relations with Moser's fibration}\label{compwise}

We first discuss an alternative approach for the proof of Theorem \ref{main-u version} via Kronheimer's homotopy fibration given in ~\cite{Kro99}.
Let  $\w$  be a symplectic form on a rational or ruled surface $M$, and let $\mT_{\w}$ be the space of symplectic forms that are isotopic to $\w$. If $\w\in \mT_u$, then $\mT_{\w}$ is the path connected component of $\mT_{u}$ containing $\w$. Let $ G_{\w}=\Symp(M,\w)\cap \Diff_0(M)$. Moser's lemma implies that $\Diff_0(M)$ acts transitively on $\mT_{\w}$, and that we have a fibration
\begin{equation}  \label{Kronheimer}
G_{\w}\to \Diff_0(M) \to \mT_{\w}.
\end{equation}
Note that the fibration~\eqref{Kronheimer} coincides with fibration~\eqref{UFIB} whenever $\mT_{\w}=\mT_{u}$, that is, if $\mT_{u}$ is path connected, or equivalently, if $\Diff_{u}$ is path connected. 
This is generally unknown, even for basic rational surfaces like $(\CC P^2, \w_{FS})$. Note that it is conjectured that $\mT_{[\w_{FS}]}$ is contractible, see~\cite[Problem 3 in Section 14.1]{MS17}.  

Let $\mA_{\w}$ be the space of almost complex structures that are compatible with some element in  $\Omega_{\omega}$ and let $P_{\w} = \{(\alpha, J) \in  \mT_{\w} \times \mA_{\w}~|~\alpha \text{~is compatible with~} J\}$.
By the  argument  in Lemma~\ref{convexfiber}, $\mA_{\w}$ and $P_{\w} $ are path connected. Moreover, $\mA_{\w}$  is a path connected component of $\mA_u$ and is homotopy equivalent to $\mT_{\w}$. In fact, $\mT_{\w}$ and $\mA_{\w}$ correspond to each other under the bijection between the sets of path connected components of $\mT_{u}$ and $\mA_{u}$ in Lemma~\ref{convexfiber}.
In particular, the same arguments as in Section~\ref{stra} imply that the sequence of maps
\begin{equation} \label{homotopy fibration for omega}
G_{\w}\to \Diff_0(M) \to \mA_{\w}.
\end{equation}
induces a long exact homotopy sequence. Let $J\in \mA_{\w}$.  Suppose we perform the $b^+=1$ $J-$compatible inflation established in Lemma \ref{Weak Inflation b1+} to $\omega$ to get a symplectic form $\w'$ compatible with $J$.  
For $\w'$  we have $\mT_{\w'}$, $\mA_{\w'}$ and the sequence of maps $G_{\w'}\to \Diff_0(M) \to \mA_{\w'}$.  Moreover, if we pick $\tilde J \in \mA_\w$ and repeat the inflation to obtain a $\widetilde{\w'}$ compatible with $\tilde J$ and cohomologous to $\w'$, by Lemma~\ref{convexfiber} $\w'$ and  $\widetilde{\w'}$ have to be isotopic.
We can prove properties analogous  to Lemma~\ref{inclusion} and Lemma~\ref{level n} for $S_{\w}$, $S_{\w'}$, $\mA_{\w}$, $\mA_{\w'}$ using the same arguments. Similarly, we  prove the analogue of Lemma ~\ref{basic} for $\pi_0(G_{\w})$, $\pi_0(G_{\w'})$, $\pi_i(\Symp(M, \w))$ and $\pi_i(\Symp(M, \w')$, with $i\geq 1$, using the following commutative diagram
\begin{equation*}
\begin{tikzcd}[row sep=small]
G_{\w} \arrow[r]\arrow[d] & \Diff_0(M) \arrow[d, equal]\arrow[r] & \mA_{\w}\arrow[d,hookrightarrow] \\ 
G_{\w'} \arrow[r] &\Diff_0(M) \arrow[r] & \mA_{\w'}
\end{tikzcd}
\end{equation*}
We can also prove the analogues of Theorem~\ref{2n-3} and Theorem~\ref{main-u version}, by performing small modifications in the proofs, mostly in notation, except that $\pi_0(\Symp(M, \w))$ and $\pi_0(\Symp(M, \w'))$ are replaced by $\pi_0(G_{\w})$ and $ \pi_0(G_{\w'})$, respectively. In order to obtain equalities between  $\pi_0(\Symp(M, \w))$ and $\pi_0(\Symp(M, \w'))$, we apply Lemma~\ref{homological action}, Proposition~\ref{-1same} and the following lemma which identifies $G_{\w}$ as the Torelli symplectic mapping class group for a rational surface. 

 \begin{lma}\label{G}
 Let $(M, \w)$ be a symplectic rational surface. Then $G_{\w}=\Symp_h(M, \w)$. 
\end{lma}
\begin{proof}
Clearly, we have 
\[  G_{\w}=\Symp (M, \w)\cap \Diff_0(M) \subset \Symp_h(M, \w).\]
By ~\cite{She10} and~\cite{LLW16},  for a rational or ruled surface $M$, we have
\[\Symp_h(M, \w) \subset \Diff_0(M).\]

Hence $Symp_h(M,\w)\subset G_{\w}=\Symp (M, \w)\cap \Diff_0(M).$
{
Therefore, $G_{\w}=\Symp_h(M, \w)$ immediately follows from the inclusions in the two directions.
Note that $\Symp_0(M, \w)\subset  G_\w$
and $\pi_0(G_{\w})=\Symp_h/\Symp_0$ is just the Torelli symplectic mapping class group. }
\end{proof} 

\subsection{Relations with embeddings of symplectic balls}
In this last section, we explain the implications of the stability theorem~\ref{FinerStabilityThm} to the stability of spaces of symplectically embedded balls. To this end, we first recall the setting proposed in~\cite{LP04} to investigate these embedding spaces.

Consider a symplectic $4$-manifold $(M,\w)$ of capacity $c_{M}$. Let $B^4(c)\subset\RR^4$ be the closed standard ball of radius $r$ and capacity $c=\pi r^2$, and let $\Symp(B^4 (c))$ be the group of symplectomorphisms of $B^{4}(c)$ that extend to some open neighbourhood of $\Symp(B^4 (c))$. Let $\Emb_{\w}(B^4(c), M)$ be the space, endowed with the $C^{\infty}$-topology, of all symplectic embeddings of $B^4(c)$ into $(M,\w)$. By definition, this space is non-empty whenever $0<c<c_{M}$. Let $\Im\Emb_{\w}(B^4(c), M)$ be the space of subsets of $M$ that are images of maps belonging to $\Emb_{\w}(B^4(c), M)$, which we topologize as the quotient
\begin{equation*} 
\Im\Emb_{\w}(B^4(c), M):= \Emb_{\w}(B^4(c), M) / \Symp(B^4 (c)).
\end{equation*}
where the group $\Symp(B^4 (c))$ acts by reparametrizations. By definition, $\Im\Emb_{\w}(B^4(c), M)$ is the space of unparametrized balls of capacity $c$ of $M$. For any choice of capacities $c<c'<c_{M}$, the restriction of embeddings to $B^{4}(c)\subset B^{4}(c')$ induces a continuous map 
\[r_{c',c}:\Emb_{\w}(B^4(c'), M)\to \Emb_{\w}(B^4(c), M)\]
We would like to investigate the connectivity of $r_{c',c}$ and, in particular, to find conditions under which $r_{c',c}$ is a weak homotopy equivalence,  {that is, we would like to understand when this map induces isomorphisms on the homotopy groups of the embedding spaces.}

To simplify the discussion, let us assume that the space $\Emb_{\w}(B^4(c), M)$ is connected. As was shown by D. McDuff in~\cite{McD98}, this is the case for all rational or ruled surfaces. Under this hypothesis, the group $\Symp(M,\w)$ acts transitively on $\Emb_{\w}(B^4(c),M)$. { After choosing a fixed embedding  $\iota_c: B^4(c)\into (M, \omega)$, denote the image of the ball by $B_{\iota_{c}}:=\iota_c(B^{4}(c))$. We obtain two fibrations} 

\begin{equation}\label{fibration embeddings}
\Symp^{\id}(M,B_{\iota_{c}})\to \Symp(M,\w)\to\Emb_{\w}(B^4(c),M)
\end{equation}
\begin{equation}\label{fibration unparametrized}
\Symp(M,B_{\iota_{c}})\to \Symp(M,\w)\to\Im\Emb_{\w}(B^4(c),M)
\end{equation}
whose fibers are, respectively,
\[\Symp^{\id}(M,B_{\iota_{c}}) = \{\phi\in\Symp(M,\w)~|~\phi|_{B_{\iota_{c}}}=\id\}\] 
and
\[\Symp(M,B_{\iota_{c}}) = \{\phi\in\Symp(M,\w)~|~\phi(B_{\iota_{c}})=B_{\iota_{c}}\}.\]
The two fibrations~\eqref{fibration embeddings} and~\eqref{fibration unparametrized} are essentially equivalent and their relation can be understood through  the evaluation fibration
\[\Symp^{\id}(M,B_{\iota_{c}}) \to \Symp(M,B_{\iota_{c}})\to \Symp(B^4(c))\simeq\Sp(4)\simeq\UU(2).\]
let $\widetilde{M_c}:=(\widetilde M, \widetilde{\w}_c)$ be the blow-up of $M$ at $B_{\iota_{c}}$.
As explained in~Section 2 of~\cite{LP04}, there is a homotopy equivalence
\begin{equation}\label{homotopy equivalence ball-fiber}
\Symp(\widetilde{M_c},\Sigma) \simeq \Symp(M,B_{\iota_{c}})
\end{equation}
where $\Symp(\widetilde{M_c},\Sigma)\subset \Symp(\widetilde{M},\widetilde{\w}_c)$ is the subgroup of symplectomorphisms of the blow-up sending the exceptional divisor $\Sigma$ to itself. 

The homotopy type of $\Symp(\widetilde{M_c},\Sigma)$ can be investigated using relative analogues of the fibration~\eqref{UFIB} and of the maps~\eqref{homotopy fibration}. To see this, let $\Diff_c(\widetilde{M},\Sigma)$ be the group of diffeomorphisms of the blow-up $\widetilde{M}$ that preserve the class $[\widetilde{\w}_c]$ and that leave the exceptional divisor $\Sigma$ invariant. Let $\Omega_c(\Sigma)$ be the space of symplectic forms cohomologous to $\widetilde{\w}_c$ and for which $\Sigma$ is symplectic. By applying a relative version of Moser's lemma, one can show that there is a fibration
\begin{equation}\label{relativeUFIB}
\Symp(\widetilde{M_c}, \Sigma)\to \Diff_c(\widetilde{M},\Sigma) \to \mT_{c}(\Sigma),
\end{equation}
Similarly, we can define the space of pairs 
\[P_c(\Sigma) =\left\{ (\w',J)~|~\w'\in\Omega_c(\Sigma),~J\text{~is compatible with~}\w',~\Sigma\text{~is~}J\text{-holomorphic} \right\}\]
and the space of compatible almost-complex structures
\[\mA_c(\Sigma) = \left\{J\text{~is compatible with some~}\w'\in\Omega_c(\Sigma)\text{~and~}\Sigma\text{~is~}J\text{-holomorphic}\right\}.\]
Then, the same arguments as in Section~\ref{stra} yield a sequence of maps
\begin{equation} \label{relative homotopy fibration}
\Symp(\widetilde{M_c}, \Sigma)\into \Diff_c(\widetilde{M},\Sigma) \to \mA_{c}(\Sigma).
\end{equation}
that induces a long exact sequence of homotopy groups. We introduce a filtration $\mA^{2n}_c(\Sigma)$ defined, as before, in terms of $J$-holomorphic spheres of self-intersections $\leq -2$. Since $\mA_c(\Sigma)$ is open in $\mA_c$, the codimension of prime subsets in $\mA_c(\Sigma)$ is given by the formula appearing in Definition~\ref{def:prime subset}. It follows that the complement of $\mA^{2n}_c(\Sigma)$ in $\mA_c(\Sigma)$ is a union of submanifolds of codimensions $\geq 2n$. Note, however, that only negative classes $S\in\mS^{\leq-2}_c$ with $S\cdot [\Sigma]\geq 0$ can be represented by $J$-holomorphic spheres for $J\in \mA_c(\Sigma)$. This implies that $\mA_c(\Sigma)$ may decompose into a strictly smaller number of strata than $\mA_c$, and that the codimension of the complement of $\mA^{2n}_c(\Sigma)$ in $\mA_c(\Sigma)$ may be strictly greater than $2n$. With this understood, we can finally prove Theorem~\ref{StabilitySingleBall}.
\begin{proof}[Proof of Theorem~\ref{StabilitySingleBall}]
Let $(M,\w)$ be a rational surface with $\chi(M)\leq 11$ of capacity $c_M$. Pick two capacities $0<c<c'<c_M$. Suppose that, after performing symplectic blow-ups of sizes $c$ and $c'$ on $(M,\w)$, we obtain symplectic forms $\widetilde{\w}_{c}$ and $\widetilde{\w}_{c'}$ that belong to the same level $n\geq 2$ chamber of $\widetilde{M}$. Write $u=[\widetilde{\w}_{c}]$ and $u'=[\widetilde{\w}_{c'}]$ and consider the diagram
\begin{equation}
\begin{tikzcd}[row sep=small]
\Symp(\widetilde{M_u},\Sigma) \arrow[r]\arrow[d] & \Diff_u(\widetilde{M},\Sigma) \arrow[d, equal]\arrow[r] & \mA_u(\Sigma)\arrow[d,hookrightarrow] \\ 
\Symp(\widetilde{M_{u'}},\Sigma) \arrow[r] & \Diff_{u'}(\widetilde{M},\Sigma) \arrow[r] & \mA_{u'}(\Sigma)
\end{tikzcd}
\end{equation}
Because $u$ and $u'$ are in the same level $n\geq 2$ chamber, Lemma~\ref{level n} implies that we have equality $\mA^{2n}_u=\mA^{2n}_{u'}$ which, in turns, immediately implies that $\mA^{2n}_u(\Sigma)=\mA^{2n}_{u'}(\Sigma)$. We also have the equality $\Diff_u(\widetilde{M},\Sigma)=\Diff_{u'}(\widetilde{M},\Sigma)$, and arguing as in the proofs of Proposition~\ref{2n-3} and of Theorem~\ref{main-u version}, we conclude that we have isomorphisms $\pi_i\Symp(\widetilde{M_{u}},\Sigma) \simeq \pi_i\Symp(\widetilde{M_{u'}},\Sigma)$ in the range $0\leq i\leq 2n-3$. From the construction of the homotopy equivalence~\eqref{homotopy equivalence ball-fiber} given in~\cite{LP04} Section 2, we conclude that the inclusions $\Symp(M,B_{\iota_{c'}})\into\Symp(M,B_{\iota_{c}})$ and  $\Symp^{\id}(M,B_{\iota_{c'}})\into\Symp^{\id}(M,B_{\iota_{c}})$ are $(2n-3)$-connected. Combining these restriction maps with the fibration~\eqref{fibration embeddings} yields a strictly commuting diagram
\begin{equation}\label{ladder embeddings}
\begin{tikzcd}[row sep=small]
\Symp^{\id}(M,B_{\iota_{c'}})\arrow[d, hookrightarrow]\arrow[r] & \Symp(M,\w)\arrow[d,equal]\arrow[r] & \Emb_{\w}(B^4(c'),M)\arrow[d,"~r_{c',c}"] \\
\Symp^{\id}(M,B_{\iota_{c}})\arrow[r]  & \Symp(M,\w)\arrow[r] & \Emb_{\w}(B^4(c),M)
\end{tikzcd}
\end{equation}
Then, the 5-lemma implies that the restriction map $r_{c',c}$ is at least $(2n-3)$-connected.
\end{proof}

In general, the stability Theorem for embeddings is not enough to compute the homotopy type of $\Emb_{\w}(B^4(c), M)$. However, there are situations in which we can combine Theorem~\ref{StabilitySingleBall} together with symmetry arguments or limit procedures to gain some understanding of the embedding space. For example, when the capacity of the symplectic ball $B^4(c)$ is small enough, we expect embedded balls to behave more or less like points in $(M,\w)$. This is the content of Corollary~\ref{cor:embeddings and stabilizers} that we now prove.

\begin{proof}[Proof of Corollary~\ref{cor:embeddings and stabilizers}]
We first assume $M\neq \CP{2}$. Suppose the capacity $c$ of the embedded balls is smaller than the symplectic area of any embedded symplectic sphere of negative self-intersection in the blow-up $(\widetilde{M_c},\widetilde{\w}_c)$. In particular, the class $E$ of the exceptional divisor $\Sigma$ has the smallest symplectic area among all exceptional classes in $\mE(\widetilde{M_c})$. By Lemma~1.2 in~\cite{P08ii}, this implies that for all $J\in\mA_u$, the class $E$ is represented by an embedded $J$-holomorphic sphere, that is, $\mA_c(\Sigma)=\mA_u$. In particular, $E\cdot A\geq 0$ for each negative class $A\in \mS_{[\widetilde{\w}_c]}^{\leq-1}$. It follows that performing inflation along the exceptional divisor $\Sigma$ decreases the symplectic area of $E$ while it increases the area of all other classes in $\mS_{[\widetilde{\w}_c]}^{\leq-1}$. This shows that the inflation along $\Sigma$ produces symplectic forms in the same level $\infty$ chamber than $\widetilde{\w}_c$. In particular, this $\infty$ chamber contains symplectic forms for which the exceptional class $E$ has arbitrarily small area $\epsilon>0$ and, by Theorem~\ref{StabilitySingleBall}, the weak homotopy type of $\Symp(\widetilde{M_c},\Sigma)$ is stable within this chamber. We can now argue as in~\cite{P08i} Section 3.1 to show that we have weak homotopy equivalences
\[\Symp(\widetilde{M_c},\Sigma)\simeq \Symp(M,B_{\iota_{c}})\simeq \Symp(M,p)\]
where $\Symp(M,p)$ is the stabilizer of a point in $M$, and that we have a commutative diagram
\begin{equation}\label{ladder stabilizer}
\begin{tikzcd}[row sep=small]
\Symp^{\id}(M,B_{\iota_{c}})\arrow[d, hookrightarrow]\arrow[r] & \Symp(M,\w)\arrow[d,equal]\arrow[r] & \Emb_{\w}(B^4(c),M)\arrow[d,"~\ev"] \\
\Symp(M,p)\arrow[r]  & \Symp(M,\w)\arrow[r] & \mathrm{SF}(M)
\end{tikzcd}
\end{equation}
where $\mathrm{SF}(M)$ is the symplectic frame bundle of $M$, where $\ev:\Emb_{\w}(B^4(c),M)\to \mathrm{SF}(M)$ is the evaluation of the derivative at the center, and where $\Symp^{\id}(M,B_{\iota_{c}})\to \Symp(M,p)$ is the inclusion. It follows that the map
\[\ev:\Emb_{\w}(B^4(c),M)\to \mathrm{SF}(M)\]
is a weak homotopy equivalence. Finally, taking the quotient by $\Sp(4)\simeq\UU(2)$ on both sides yields the weak homotopy equivalence
\[\Im\Emb_{\w}(B^4(c),M)\simeq M\]
This concludes the proof when $M\neq\CP{2}$. 

In the case $M=\CP{2}$, the only difference is that Lemma~1.2 in~\cite{P08ii} does not apply. However, because we are assuming that the capacity $c$ of the embedded balls is smaller than the symplectic area of \emph{any} embedded symplectic sphere of negative self-intersection in the one-point blow-up $\CP{2}\#\overline{\mathrm{\CP}}^2$, it still follows that the class $E$ is symplectically indecomposable and that it is represented by an embedded $J$-sphere for any choice of almost-complex structure $J\in\mA_c$. The rest of the argument is the same as before.
\end{proof}

We end this paper by stating a conjecture that generalizes Corollary~\ref{cor:embeddings and stabilizers}. Given any symplectic 4-manifold $(M, \w)$, we consider the embedding of $k$ disjoint balls of sizes $c_1\ldots, c_k$.  We denote the space of such embeddings by $\Emb_k(c_i)$. Let $\Conf_k(M)$ be the configuration space of ordered $k$ points in $M$. There is a forgetful map 
\[\sigma:\Emb_k(c_i)\to \Conf_k(M),\]
defined by evaluating an embedding at the center of each of the balls. This map allows for a comparison of the space of images $\Im\Emb_k(c_i)$ with $\Conf_k(M)$ through a diagram of fibrations. In~\cite{LWnote}, there is a conjectural higher homotopical generalization of Biran's ball packing stabilization theorem, which can be viewed as a $k$-fold version of Corollary~\ref{cor:embeddings and stabilizers}. 
\begin{conj}[Conjecture 5.2 in~\cite{LWnote}]\label{conj:conf}
For any $m\in\ZZ^+$, there exists an $\epsilon(m)>0$ such that the homotopy groups $\pi_i\Im\Emb_k(c_i)$ and $\pi_i \Conf_k(M)$ are isomorphic for all $0\leq i\leq m$ whenever $c_i<\frac{\epsilon(m)}{k}$.
\end{conj}

One can use our stability Theorem~\ref{main-u version} together with the proof of Corollary~\ref{cor:embeddings and stabilizers} to show that this conjecture is true for $M=\CC P^2$ with $k\leq 8$. For $k\ge 9$, it follows from~\cite{McD08} and~\cite{LLW16} that $\pi_1(\Symp(\CC P^2\# k\overline{ \CC P^2},\w))$ is finitely generated and that its rank is bounded by $\frac{k(k+1)}{2}$. Hence, it seems possible to prove stability at the $\pi_1$ level using techniques similar to the ones presented in the present manuscript. It is reasonable to speculate that a version of Theorem~\ref{main-u version} still holds for ${ \CC P^2}\# k\overline{ \CC P^2}$, $k>9$ (cf. Remark~\ref{12}) and hence that Conjecture~\ref{conj:conf} holds for $M=\CC P^2$ with any number of embedded balls.



\section*{Declarations}

\subsection*{Conflict of interest} All authors declare that they have no conflicts of interest.
\subsection*{Availability of data} No datasets were generated or analyzed for the current work.
\subsection*{Authors' contributions} The authors contributed equally to this work.


\end{document}